\documentclass{mathresearch}

\allowdisplaybreaks[2]

\newcommand{\scramble}{\mathcal{S}}

\newcommand{\graph}{\mathcal{G}}
\newcommand{\indset}{\mathcal{I}}

\DeclareMathOperator{\KG}{KG} 
\DeclareMathOperator{\gon}{gon} 
\DeclareMathOperator{\sn}{sn} 
\DeclareMathOperator{\tw}{tw} 
\DeclareMathOperator{\ind}{\alpha} 
\DeclareMathOperator{\dgr}{\delta} 
\DeclareMathOperator{\rk}{rk} 
\DeclareMathOperator{\clq}{\omega} 

\newcommand{\floor}[1]{\left\lfloor #1 \right\rfloor}
\newcommand{\ceil}[1]{\left\lceil #1 \right\rceil}
\newcommand{\ord}[1]{\left|\left|#1\right|\right|}
\usepackage{physics} 

\newcommand{\ZZ}{\mathbb{Z}}
\newcommand{\RR}{\mathbb{R}}
\newcommand{\NN}{\mathbb{N}}

\newcommand{\oeisref}[2][\empty]{\href{https://oeis.org/#2}{\underline{#2}}\ifx\empty#1\else\ (#1)\fi}

\usepackage[left=3.2cm, right=3.2cm]{geometry}

\title{On the Gonality of Kneser Graphs}
\author[Ballinas]{Luis A. Ballinas}
\address[L.\ Ballinas]{Department of Mathematics, California State University, Fullerton, CA 92831
}
\email{LuisA.Ballinas@csu.fullerton.edu}

\author[Caine]{Willoughby Caine}
\address[W.\ Caine]{Department of Computer Science and Mathematics, Fort Valley State University, Fort Valley, GA 31030
}
\email{cainew1@unlv.nevada.edu}

\author[Hopkins]{D. Blake Hopkins}
\address[D.\ Hopkins]{UT Tyler Department of Mathematics, University of Texas at Tyler, Tyler, TX 75707
}
\email{d.blake.hopkins@uky.edu}

\author[Rivera]{Doel Rivera Laboy}
\address[D.\ Rivera]{Department of Mathematics, University of Kentucky
}
\email{doel.riveralaboy@uky.edu}

\begin{document}

\begin{abstract}

The Kneser graphs $\KG(n,k)$ are a classically studied family of graphs. One known invariant of graphs is gonality (also called divisorial gonality), which is the minimum degree of a rank 1 divisor on the graph. 

Using known bounds on gonality of simple, connected graphs, one may obtain that the gonality of $\KG(n,k)$ is bounded above by $\binom{n-1}{k}$.
In 2014, Harvey and Wood showed that the treewidth (a lower bound on gonality) for $\KG(n,k)$ is $\binom{n-1}{k}-1$ for $n\geq 4k^2-3k+2$. In this paper, using scramble number, another lower bound on gonality, we improve this polynomial bound and show that the gonality of $\KG(n,k)$ is exactly $\binom{n-1}{k}$ for $n\geq \frac{3k^2+k+2}{2}$, and conjecture an even stricter polynomial bound using the uniform edge scramble. We then extend our argument to the family of generalized Kneser Graphs, computing the scramble number and gonality using the same polynomial bound.

\end{abstract}

\maketitle

\section{Introduction}

Chip firing games on graphs form a combinatorial analog of divisor theory on algebraic curves. 
These games have been studied through the lens of other areas such as structural combinatorics \cite{bak1988self} and the abelian sandpile model \cite{bjorner1991chip}.
Baker and Norine later introduced the question of gonality in \cite{Baker08}.
Computing gonality is known to be an NP-hard problem and an APX-hard problem \cite{gijswijt2020computing}, meaning that it is difficult to even approximate.
An upper bound can be achieved by finding an appropriate ``winning configuration", but it is more difficult to obtain lower bounds.
To this end, several graph invariants have been studied which provide a lower bound on gonality.

A significant step in this direction was obtained in \cite{van_Dobben_de_Bruyn_2020}, in which the authors show that the gonality of a graph $G$ is bounded below by a well studied graph invariant known as treewidth \cite{ROBERTSON}. 
The most powerful lower bound known to date is the scramble number of a graph, introduced in \cite{HJJS}, where it was proved that the scramble number is at least as large as treewidth, and is no larger than gonality. It is known that there exist graphs that have scramble number strictly greater than treewidth.

In computer science, treewidth has been studied for many classes of graphs, as it is a measure of how well certain algorithms can run on the graph \cite{voigt2016tree}. 
However, scramble number and gonality has been studied mostly with the focus of understanding divisor theory on these graphs. This leaves space for families of graphs 
whose treewidth may have been studied but not their gonality, and vice versa. 
One such family, the focus of this paper, is the family of Kneser graphs (denoted $\KG(n,k)$) and generalized Kneser graphs (denoted $\KG(n,k,t)$).
These graphs were first studied in 1956 by Martin Kneser \cite{kneser1955aufgabe}, and they have become a graph family well studied for their combinatorial structure. 
The treewidth of these graphs has been studied in \cite{harvey2013treewidthknesergrapherdhoskorado}, and \cite{liu2021treewidthgeneralizedknesergraphs}, where they compute the treewidth as long as the parameter $n$ is ``large enough" in comparison to the other parameters. They prove that if $n\geq 4k^2-3k+2$, then $\text{tw}(\KG(n,k)) =\binom{n-1}{k}-1 = |V(\KG(n,k))|-\ind-1$, where $\ind$ denotes the independence number of $\KG(n,k)$ and $V(\KG(n,k))$ denotes the set of vertices of $\KG(n,k)$. The independence number of a graph is the size of the largest independent set (see  \cref{sec: background}). 

In this article we first establish that, for scramble number and gonality, there also must exist some notion of ``$n$ being large enough" for the parameters to behave as expected. 
We do this by looking at the case where $n$ is as close to $k$ as possible while remaining connected in \cref{prop: odd graph gonality}. 

The later half of the paper will focus on proving our main results, which show that the scramble number is strictly greater than treewidth for computing the gonality of an infinite set of Kneser graphs. In section \ref{sec: proof} we prove
\begin{restatable}{theorem}{mainthm}
\label{thm: 3k^2+k lower bound}
    Let $n \geq \frac{3k^2 + k + 2}{2}$. Then $$\sn(\KG(n,k)) = \gon(\KG(n,k)) = \binom{n-1}{k}.$$
\end{restatable}
 This theorem shows that both the scramble number and the gonality of Kneser graphs are exactly $1$ more than the treewidth for all of the cases which the treewidth had been computed and $k\geq2$. That is, this is a family of graphs where the treewidth is exactly $1$ less than both the scramble number and the gonality. Additionally, since $4k^2-3k+2>\frac{3k^2 + k + 2}{2}$, this yields an infinite subset of Kneser graphs whose treewidth was unknown but whose scramble number and gonality we were able to compute.

We prove Theorem \ref{thm: 3k^2+k lower bound} using two main tools.
The first is a sufficient condition for computing both scramble number and gonality, proven in \cite{ECHAVARRIA202243}. The second is using derivatives involving the Gamma function, a smooth extension of the factorial function, in order to show the sufficient condition is satisfied when our polynomial bound is satisfied.

In Section \ref{sec: gen kneser}, we then generalize our argument to 
the family of generalized Kneser graphs, $\KG(n,k,t)$, studied in \cite{liu2021treewidthgeneralizedknesergraphs} and \cite{ajc-v86-p477}. Using the same polynomial bound, we computed the scramble number and gonality of $\KG(n,k,t)$. This yields another large collection of graphs whose treewidth was unknown, but whose scramble number and gonality we computed. 
\begin{restatable}{theorem}{genkneser}\label{thm: gen kneser}
	For any $k\geq t \geq 1$, if $n\geq \frac{3k^2+k+2}{2}$, then $$\sn(\KG(n,k,t)) =\gon(\KG(n,k,t)) = \binom{n}{k}-\binom{n-t}{k-t}.$$
\end{restatable}

\begin{restatable}{corollary}{infsubsets}\label{coll: inf subs}
    There exist infinite collections of  generalized Kneser Graphs $\KG(n,k,t)$ such that
    \begin{enumerate}
        \item $\gon(\KG(n,k,t)) = \sn(\KG(n,k,t)) =\tw(\KG(n,k,t))+1$ and
        \item the treewidth of $\KG(n,k,t)$ is unknown, but $\gon(\KG(n,k,t)) = \sn(\KG(n,k,t)) = \binom{n}{k}-\binom{n-t}{k-t}.$
    \end{enumerate}
\end{restatable}
Finally, a property of Kneser graphs that has been studied is the $\lambda_p$-optimality \cite{ESFAHANIAN1988195,BALBUENA2019258}. This property relates to computing the scramble number of special scrambles known as ``uniform scrambles" \cite{cenek2023uniformscramblesgraphs}. In \cite{BALBUENA2019258}, the authors determined that Kneser graphs are $\lambda_2$-optimal if $k=2$. However, it is currently unknown if Kneser graphs are $\lambda_2$-optimal if $k>2$. This remains an open question; however, in Section \ref{sec: conj improvement} we prove the following:
\begin{restatable}{theorem}{loptimal}
    If $\KG(n,k)$ is $\lambda_2$-optimal and $n\geq \frac{3k^2-3k+10}{2}$, then $$\sn(\KG(n,k))=\gon(\KG(n,k))=\binom{n-1}{k}.$$
\end{restatable}

\subsection*{Outline of the paper} In \cref{sec: background}, we introduce all the required definitions and notation for understanding the chip firing game and the Kneser graphs. In \cref{ssec: upper bounds for gonality},
we discuss the common methods for obtaining upper bounds for gonality. Then, in \cref{ssec: scramble num}, we define the main tool for this paper, scramble number. In Section \ref{sec: main work}, we 
establish our candidate bounds and proceed to prove our main theorem for Kneser graphs, using calculus techniques provided in Section \ref{sec: calc}. We then extend this result to generalized Kneser graphs in Section \ref{sec: gen kneser}. Afterwards, in Section \ref{sec: conj improvement}, we consider the possible improvements to our main results under the assumption of $\lambda_2$-optimality. In the last section, we discuss possible improvements to our bounds and future directions this problem may take.

\subsection*{Acknowledgments} The authors were supported by NSF grant 215034 at the Combinatorics and Coding Theory in the Tropics REU. Special thanks to Dr. Pamela Harris and Dr. Fernando Piñero for procurement of the grant. We thank Dr. William L. Blair and Dr. David Jensen for their insightful feedback on the earlier drafts. 

\section{Background}
\label{sec: background}

\subsection{Kneser graphs}
\label{ssec: Kneser graphs}
In 1956, Martin Kneser investigated the \emph{Kneser graphs} $\KG(n,k)$ \cite{kneser1955aufgabe}, whose vertices correspond to the $k$-element subsets of a set of $n$ elements and where two vertices are adjacent if and only if the two corresponding sets are disjoint. That is, if $v_A$ and $v_B$ are vertices, there is an edge between them if their corresponding sets satisfy $|A\cap B|= \emptyset$.
Throughout this paper, we let $[n]=\{1,2,\dots,n\}$ and denote by $\binom{[n]}{k}$ the set of all $k$-subsets of $[n]$, so that vertices of $\KG(n,k)$ correspond to sets in $\binom{[n]}{k}$.
Kneser graphs are highly structured and many of their properties are well-studied.
Kneser graphs are simple; that is, they contain no multiedges \cite{ekinci2018super}. 
They are connected if and only if $n>2k$. 

The \emph{degree} of a vertex in a graph $\graph$ is the number of edges connected to it. 

We let $V(\graph)$ denote the set of vertices in a graph $\graph$, and $|V(\KG(n,k))| = \binom{n}{k}$ and every vertex $v \in V(\KG(n,k))$ has degree $\binom{n-k}{k}$ \cite{ekinci2018super}.

A \emph{clique} of a graph $\graph$ is an induced subgraph on $\graph$ that is complete. The \emph{clique number} $\clq(\graph)$ is the size of the largest clique that can be induced on $\graph$.

Brouwer and Schrijver \cite{brouwer1979uniform} showed that when $n<ck$, $\KG(n,k)$ does not contain cliques of size $c$. As a result, we have the following lemma:

\begin{lemma}[\cite{brouwer1979uniform}]
\label{lem:kneser clique number}
    The clique number of  $\KG(n,k)$ is $$\clq(\KG(n,k)) = \floor{\frac{n}{k}}.$$
\end{lemma}

We say two vertices in $V(\graph)$ are \emph{independent} if they are not adjacent, and call a set of pairwise independent vertices an \emph{independent set in $\graph$}. 
The \emph{independence number} $\ind(\graph)$ is the size of the largest possible independent set in $\graph$.

We are interested in chip-firing on connected graphs, so we consider Kneser graphs with $n>2k$. 
Using the Erd\H{o}s-Ko-Rado Theorem, Liu, Cao, and Lu gave the following result:

\begin{lemma}[\cite{liu2021treewidthgeneralizedknesergraphs}*{Theorem 3.1}]
\label{lem:n >= 2k independence}
    For $n > 2k$, the independence number of $\KG(n,k)$ is $$\ind(\KG(n,k)) = \binom{n-1}{k-1}.$$
\end{lemma}

In \cref{fig:prop2.2example}, we show a maximum independent set on the complete graph $K_4$ and the Petersen graph, both of which are the Kneser graphs $\KG(4,1)$ and $\KG(5,2)$ respectively.

\begin{figure}[h]\centering
    \begin{tikzpicture}[scale=0.25]
        \node[draw, circle, fill, inner sep=1.5pt] (S1) at (4.24,4.24) {};
        \node[draw, circle, fill, inner sep=1.5pt] (S2) at (4.24,-4.24) {};
        \node[draw, circle, fill, inner sep=1.5pt] (S3) at (-4.24,4.24) {};
        \node[draw, circle, fill, inner sep=1.5pt] (S4) at (-4.24,-4.24) {};
        \node[draw, circle, red, thick, inner sep=3pt] at (S4) {};

        \draw[-, thick] (S1) to (S2) to (S3) to (S4) to (S1) to (S3) to (S2) to (S4);
    
        \node[draw, circle, fill, inner sep=1.5pt] (PO1) at (20,6) {};
        \node[draw, circle, fill, inner sep=1.5pt] (PO2) at (25.7,1.85) {};
        \node[draw, circle, fill, inner sep=1.5pt] (PO3) at (23.53,-4.85) {};
        \node[draw, circle, fill, inner sep=1.5pt] (PO4) at (16.47,-4.85) {};
        \node[draw, circle, fill, inner sep=1.5pt] (PO5) at (14.29,1.85) {};
    
        \draw[-,thick] (PO1) to (PO2) to (PO3) to (PO4) to (PO5) to (PO1);
        
        \node[draw, circle, fill, inner sep=1.5pt] (PI1) at (20,3) {};
        \node[draw, circle, fill, inner sep=1.5pt] (PI2) at (22.85,0.93) {};
        \node[draw, circle, fill, inner sep=1.5pt] (PI3) at (21.76,-2.43) {};
        \node[draw, circle, fill, inner sep=1.5pt] (PI4) at (18.24,-2.43) {};
        \node[draw, circle, fill, inner sep=1.5pt] (PI5) at (17.15,0.93) {};
    
        \draw[-,thick] (PI1) to (PI3) to (PI5) to (PI2) to (PI4) to (PI1);
    
        \draw[-,thick] (PO1) to (PI1);
        \draw[-,thick] (PO2) to (PI2);
        \draw[-,thick] (PO3) to (PI3);
        \draw[-,thick] (PO4) to (PI4);
        \draw[-,thick] (PO5) to (PI5);

        \node[draw, circle, red, thick, inner sep=3pt] at (PO5) {};
        \node[draw, circle, red, thick, inner sep=3pt] at (PO2) {};
        \node[draw, circle, red, thick, inner sep=3pt] at (PI3) {};
        \node[draw, circle, red, thick, inner sep=3pt] at (PI4) {};
    \end{tikzpicture}
    \caption{$\KG(4,1)$ (left) and $\KG(5,2)$ (right) with maximum independent sets of vertices circled in red.}
    \label{fig:prop2.2example}
\end{figure}
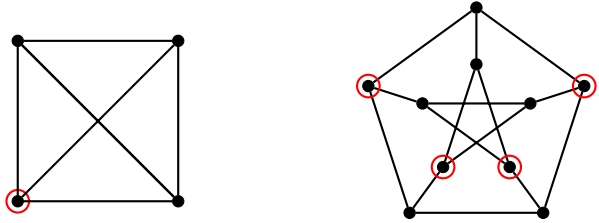

We now address the graph invariant that we focus on, gonality (sometimes called ``divisorial gonality"). 
We define gonality in the context of \emph{chip firing}, where we imagine placing a number of poker chips on each vertex of a graph, and perform the operation of chip firing to ``win the game", all of which we define formally below.

\subsection{The Chip Firing Game and Gonality}
\label{ssec: Chip Firing}

As mentioned, the chip firing game has been approached from various mathematical perspectives, including via divisor theory.
With this in mind, as we define the notation used throughout the paper in this \lcnamecref{ssec: Chip Firing}, we include alternate common terminology for completeness.

Let $\graph$ be a connected graph on a finite number of vertices.
To set up the chip firing game, we assign each vertex an integer number of poker chips.
We call an assignment $D : V(\graph) \to \ZZ$ a \emph{divisor} (also called a \emph{chip arrangement}).
Intuitively, we imagine placing $D(v)$ poker chips on vertex $v$ when $D(v) \geq 0$, and we say that $v$ is ``in debt'' if $D(v) < 0$. We may then represent these divisors as formal $\ZZ$-linear combinations of the vertices. 
\begin{definition}
    We denote a divisor $D$ on a graph $\graph$ by $\displaystyle D = \sum_{\mathclap{v \in V(\graph)}}D(v)\cdot v$ and say $D$ is \emph{effective} if $D(v) \geq 0$ for all $v \in V(\graph)$.
    The \emph{degree} $\deg(D)$ of $D$ is $\displaystyle \deg(D) = \sum_{\mathclap{v \in V(\graph)}}D(v)$.
\end{definition} 

Effectiveness (having no debt) determines the win condition for the chip firing game, and degree informs ``score,'' in an informal sense, by giving a notion of size to divisors.
Gonality investigates the minimum possible degree for winning the game after certain initial conditions.

The chip firing game is played by performing chip firing moves; respective to a given divisor $D$, a chip firing move is performed by choosing one vertex $v$ and moving one chip along each edge connected to $v$.

In the general case, each vertex $w$ adjacent to $v$ gains as many chips as there are edges connecting $v$ to $w$, and $v$ loses the total number of chips received by all of the neighboring vertices. 
For simple graphs, there is at most one edge between any two vertices, so a chip firing move constitutes increasing $D(w)$ by 1 for every vertex $w$ adjacent to $v$ and decreasing $D(v)$ by the number of vertices adjacent to $v$.

This operation results in a new chip arrangement.
We note that this notion may be generalized to firing sequences of vertices, as well as firing sets of vertices simultaneously. 
\begin{definition}
    A \emph{chip firing move} maps a divisor $D$ on a simple graph $\graph$ and a vertex $v \in V(\graph)$ to the divisor $$D' = D - |\mathcal{N}(v)|v + \sum_{w \in \mathcal{N}(v)}w,$$ where $\mathcal{N}(v) = \{w \in V(\graph) : w \text{ is adjacent to } v\}$ is the neighborhood of $v$.
    When there is no ambiguity, we say we \emph{fire} $v$ to represent transforming $D$ into $D'$, and call $D'$ the \emph{result} of firing $v$.
\end{definition}

Chip firing is invertible as we are able to undo the effect of firing $v$ by firing every vertex other than $v$, giving us an equivalence relation. We say two divisors $D_1$ and $D_2$ are \emph{equivalent} if there exists a series of chip firing moves starting from $D_1$ and yielding $D_2$ as a result.
Notably, the degree of a divisor is invariant under chip firing, so if $D_1$ is equivalent to $D_2$, then $\deg(D_2) = \deg(D_1)$.
The converse is not necessarily true.

We now explore the modification that occurs to a starting chip arrangement before playing, and its relevance to the game.

\begin{definition}
    The \emph{rank} $\rk(D)$ of a divisor $D$ is the maximum number of chips that can be removed across any number of vertices such that the resulting divisor still has an effective divisor in its equivalence class. 
    That is, 
    $\rk(D)$ is the largest nonnegative integer $r$ such that for every effective divisor $E$ on $\graph$ with $\deg(E) = r$, then $D - E$ is equivalent to an effective divisor.
    If there is no effective divisor in the equivalence class of $D$, $\rk(D) := -1$.
\end{definition}

In the context of the chip firing game, we imagine a thief steals $r$ chips from the starting chip arrangement $D$. If the game can be won starting from the resulting chip arrangement, regardless of where the thief stole the chips from, then $\rk(D) \geq r$. 

Summarizing, the chip firing game is composed of three parts:
\begin{description}
    \item[Setup] Choose a starting chip arrangement and remove $r$ chips.
    \item[Action] Perform chip firing moves.
    \item[Goal] Have all $D(v) \geq 0$. That is, to eliminate all debt.
\end{description}

Informally, the gonality of a graph is the minimum number of chips required to construct a chip arrangement such that if a thief steals a \emph{single} chip from any vertex, the resulting game is still winnable.
This is the case when $r = 1$, so we define a winning divisor as one with rank at least 1.

\begin{definition} \label{def:gonality}
    A \emph{winning divisor} is a divisor $D$ with $\rk(D) \geq 1$.
    The \emph{gonality} $\gon(\graph)$ of a graph $\graph$ is the minimum possible degree of a winning divisor on $\graph$: $\gon(\graph) = \min\{\deg(D) : \rk(D) \geq 1\}$.
\end{definition}

\section{Known Bounds on Gonality}
\label{sec: gonality}

In studying the gonality of graphs, it is helpful to use the structure of the graphs to find various bounds on gonality. For example, if the number of vertices of a simple graph is known, a viable upper bound for gonality would be given by a starting chip arrangement of one chip on each vertex. In the next section, we discuss the known upper bounds for the family of Kneser graphs $\KG(n,k)$.

\subsection{Upper Bounds}
\label{ssec: upper bounds for gonality}

Consider a graph $\graph$. 
Because gonality is a minimum, providing an upper bound is a particularly straightforward task. If you provide a divisor $D$ that is winning, then $\gon(\graph) \leq \deg(D)$. 
A divisor is winning if, for any vertex $v \in V(\graph)$, the divisor $D-v$ (formed by removing one chip from one vertex of $D$) is equivalent to an effective divisor.
We employ this technique in the proof of \cref{prop: odd graph gonality} and provide an example below.

Let $\indset \subseteq V(\graph)$ be an independent set in $\graph$ and let $D = \sum_{v \in \indset}v$.
Then for any $v \in V(\graph)$, if $v \notin \indset$, then $D - v$ is effective; otherwise, $v \in \indset$ and so $v$ is only adjacent to vertices with chips on them, meaning that firing every vertex in $V(\graph) \setminus \{v\}$ transfers chips to $v$ without creating debt, and thus $D - v$ is equivalent to an effective divisor.
For all vertices $v$ in $\graph$, $D - v$ is equivalent to an effective divisor, so $\gon(\graph) \leq \deg(D) = |V(\graph)| - |\mathcal{I}|$. Recall that $\alpha(G)$ is the largest possible size for an independent set, hence $\gon(\graph) \leq |V(\graph)| -\alpha(G)$. 
Deveau, Jensen, Kainic, and Mitropolsky use this very technique in \cite{deveau2016gonality} to provide an upper bound for gonality for simple graphs. See Figure \ref{fig:petersen degree 6 example} for an example of the winning divisor obtained by the completement of an independent set in $\KG(5,2)$.

\begin{lemma}[\cite{deveau2016gonality}*{Proposition 3.1}]
\label{lem: gon <= V - alpha}
    For any simple, connected graph $\graph$, $$\gon(\graph) \leq |V(\graph)|-\ind(\graph).$$
\end{lemma}

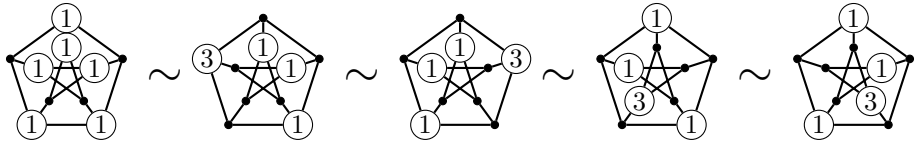
\begin{figure}[h]\centering
    \begin{tikzpicture}[scale=0.13]
        \node[draw, circle, inner sep=1pt] (P1O1) at (0,6) {1};
        \node[draw, circle, fill, inner sep=1pt] (P1O2) at (5.7,1.85) {};
        \node[draw, circle, inner sep=1pt] (P1O3) at (3.53,-4.85) {1};
        \node[draw, circle, inner sep=1pt] (P1O4) at (-3.53,-4.85) {1};
        \node[draw, circle, fill, inner sep=1pt] (P1O5) at (-5.71,1.85) {};
    
        \draw[-,thick] (P1O1) to (P1O2) to (P1O3) to (P1O4) to (P1O5) to (P1O1);
        
        \node[draw, circle, inner sep=1pt] (P1I1) at (0,3) {1};
        \node[draw, circle, inner sep=1pt] (P1I2) at (2.85,0.93) {1};
        \node[draw, circle, fill, inner sep=1pt] (P1I3) at (1.76,-2.43) {};
        \node[draw, circle, fill, inner sep=1pt] (P1I4) at (-1.76,-2.43) {};
        \node[draw, circle, inner sep=1pt] (P1I5) at (-2.85,0.93) {1};
    
        \draw[-,thick] (P1I1) to (P1I3) to (P1I5) to (P1I2) to (P1I4) to (P1I1);
    
        \draw[-,thick] (P1O1) to (P1I1);
        \draw[-,thick] (P1O2) to (P1I2);
        \draw[-,thick] (P1O3) to (P1I3);
        \draw[-,thick] (P1O4) to (P1I4);
        \draw[-,thick] (P1O5) to (P1I5);

        \node (S1) at (10,0) {\huge $\sim$};

        \node[draw, circle, fill, inner sep=1pt] (P2O1) at ([shift={(20,0)}]P1O1) {};
        \node[draw, circle, fill, inner sep=1pt] (P2O2) at ([shift={(20,0)}]P1O2) {};
        \node[draw, circle, inner sep=1pt] (P2O3) at ([shift={(20,0)}]P1O3) {1};
        \node[draw, circle, fill, inner sep=1pt] (P2O4) at ([shift={(20,0)}]P1O4) {};
        \node[draw, circle, inner sep=1pt] (P2O5) at ([shift={(20,0)}]P1O5) {3};
    
        \draw[-,thick] (P2O1) to (P2O2) to (P2O3) to (P2O4) to (P2O5) to (P2O1);
        
        \node[draw, circle, inner sep=1pt] (P2I1) at ([shift={(20,0)}]P1I1) {1};
        \node[draw, circle, inner sep=1pt] (P2I2) at ([shift={(20,0)}]P1I2) {1};
        \node[draw, circle, fill, inner sep=1pt] (P2I3) at ([shift={(20,0)}]P1I3) {};
        \node[draw, circle, fill, inner sep=1pt] (P2I4) at ([shift={(20,0)}]P1I4) {};
        \node[draw, circle, fill, inner sep=1pt] (P2I5) at ([shift={(20,0)}]P1I5) {};
    
        \draw[-,thick] (P2I1) to (P2I3) to (P2I5) to (P2I2) to (P2I4) to (P2I1);
    
        \draw[-,thick] (P2O1) to (P2I1);
        \draw[-,thick] (P2O2) to (P2I2);
        \draw[-,thick] (P2O3) to (P2I3);
        \draw[-,thick] (P2O4) to (P2I4);
        \draw[-,thick] (P2O5) to (P2I5);

        \node (S2) at (30,0) {\huge $\sim$};

        \node[draw, circle, fill, inner sep=1pt] (P3O1) at ([shift={(20,0)}]P2O1) {};
        \node[draw, circle, inner sep=1pt] (P3O2) at ([shift={(20,0)}]P2O2) {3};
        \node[draw, circle, fill, inner sep=1pt] (P3O3) at ([shift={(20,0)}]P2O3) {};
        \node[draw, circle, inner sep=1pt] (P3O4) at ([shift={(20,0)}]P2O4) {1};
        \node[draw, circle, fill, inner sep=1pt] (P3O5) at ([shift={(20,0)}]P2O5) {};
    
        \draw[-,thick] (P3O1) to (P3O2) to (P3O3) to (P3O4) to (P3O5) to (P3O1);
        
        \node[draw, circle, inner sep=1pt] (P3I1) at ([shift={(20,0)}]P2I1) {1};
        \node[draw, circle, fill, inner sep=1pt] (P3I2) at ([shift={(20,0)}]P2I2) {};
        \node[draw, circle, fill, inner sep=1pt] (P3I3) at ([shift={(20,0)}]P2I3) {};
        \node[draw, circle, fill, inner sep=1pt] (P3I4) at ([shift={(20,0)}]P2I4) {};
        \node[draw, circle, inner sep=1pt] (P3I5) at ([shift={(20,0)}]P2I5) {1};
    
        \draw[-,thick] (P3I1) to (P3I3) to (P3I5) to (P3I2) to (P3I4) to (P3I1);
    
        \draw[-,thick] (P3O1) to (P3I1);
        \draw[-,thick] (P3O2) to (P3I2);
        \draw[-,thick] (P3O3) to (P3I3);
        \draw[-,thick] (P3O4) to (P3I4);
        \draw[-,thick] (P3O5) to (P3I5);

        \node (S2) at (50,0) {\huge $\sim$};

        \node[draw, circle, inner sep=1pt] (P4O1) at ([shift={(20,0)}]P3O1) {1};
        \node[draw, circle, fill, inner sep=1pt] (P4O2) at ([shift={(20,0)}]P3O2) {};
        \node[draw, circle, inner sep=1pt] (P4O3) at ([shift={(20,0)}]P3O3) {1};
        \node[draw, circle, fill, inner sep=1pt] (P4O4) at ([shift={(20,0)}]P3O4) {};
        \node[draw, circle, fill, inner sep=1pt] (P4O5) at ([shift={(20,0)}]P3O5) {};
    
        \draw[-,thick] (P4O1) to (P4O2) to (P4O3) to (P4O4) to (P4O5) to (P4O1);
        
        \node[draw, circle, fill, inner sep=1pt] (P4I1) at ([shift={(20,0)}]P3I1) {};
        \node[draw, circle, fill, inner sep=1pt] (P4I2) at ([shift={(20,0)}]P3I2) {};
        \node[draw, circle, fill, inner sep=1pt] (P4I3) at ([shift={(20,0)}]P3I3) {};
        \node[draw, circle, inner sep=1pt] (P4I4) at ([shift={(20,0)}]P3I4) {3};
        \node[draw, circle, inner sep=1pt] (P4I5) at ([shift={(20,0)}]P3I5) {1};
    
        \draw[-,thick] (P4I1) to (P4I3) to (P4I5) to (P4I2) to (P4I4) to (P4I1);
    
        \draw[-,thick] (P4O1) to (P4I1);
        \draw[-,thick] (P4O2) to (P4I2);
        \draw[-,thick] (P4O3) to (P4I3);
        \draw[-,thick] (P4O4) to (P4I4);
        \draw[-,thick] (P4O5) to (P4I5);

        \node (S2) at (70,0) {\huge $\sim$};

        \node[draw, circle, inner sep=1pt] (P5O1) at ([shift={(20,0)}]P4O1) {1};
        \node[draw, circle, fill, inner sep=1pt] (P5O2) at ([shift={(20,0)}]P4O2) {};
        \node[draw, circle, fill, inner sep=1pt] (P5O3) at ([shift={(20,0)}]P4O3) {};
        \node[draw, circle, inner sep=1pt] (P5O4) at ([shift={(20,0)}]P4O4) {1};
        \node[draw, circle, fill, inner sep=1pt] (P5O5) at ([shift={(20,0)}]P4O5) {};
    
        \draw[-,thick] (P5O1) to (P5O2) to (P5O3) to (P5O4) to (P5O5) to (P5O1);
        
        \node[draw, circle, fill, inner sep=1pt] (P5I1) at ([shift={(20,0)}]P4I1) {};
        \node[draw, circle, inner sep=1pt] (P5I2) at ([shift={(20,0)}]P4I2) {1};
        \node[draw, circle, inner sep=1pt] (P5I3) at ([shift={(20,0)}]P4I3) {3};
        \node[draw, circle, fill, inner sep=1pt] (P5I4) at ([shift={(20,0)}]P4I4) {};
        \node[draw, circle, fill, inner sep=1pt] (P5I5) at ([shift={(20,0)}]P4I5) {};
    
        \draw[-,thick] (P5I1) to (P5I3) to (P5I5) to (P5I2) to (P5I4) to (P5I1);
    
        \draw[-,thick] (P5O1) to (P5I1);
        \draw[-,thick] (P5O2) to (P5I2);
        \draw[-,thick] (P5O3) to (P5I3);
        \draw[-,thick] (P5O4) to (P5I4);
        \draw[-,thick] (P5O5) to (P5I5);
    \end{tikzpicture}
    \caption{A winning degree $6$ divisor on $\KG(5,2)$.}
    \label{fig:petersen degree 6 example}
\end{figure}

\begin{corollary} 
\label{lem:gon < n - alpha for simple graphs} 
    Let $n>2k$. Then $$\gon(\KG(n,k))\leq\binom{n-1}{k}.$$
\end{corollary}
\begin{proof}
    As $|V(\KG(n,k))|=\binom{n}{k}$ and $\ind(\KG(n,k))=\binom{n-1}{k-1}$, we have \begin{equation*}\gon(\KG(n,k)) \leq |V(\KG(n,k))| - \ind(\KG(n,k)) = \binom{n}{k}-\binom{n-1}{k-1} = \binom{n-1}{k}.\qedhere\end{equation*}
\end{proof}

We note that this upper bound is not always optimal; consider $\KG(5,2)$, the Petersen graph. \cref{lem:gon < n - alpha for simple graphs} yields $\gon(\KG(5,2)) \leq 6$; however, the divisor defined by placing one chip on each of the vertices in a specific maximum independent set is a winning chip arrangement. See Figure \ref{fig:petersen degree 4 example} for a visual of the winning divisor and each equivalent divisor that can be used to place a chip anywhere on the graph. Altogether, this implies that $\gon(\KG(5,2)) \leq 4<6$. 

\begin{figure}\centering
    \begin{tikzpicture}[scale=0.16]
        \node[draw, circle, fill, inner sep=1pt] (P1O1) at (0,6) {};
        \node[draw, circle, inner sep=1pt] (P1O2) at (5.7,1.85) {1};
        \node[draw, circle, fill, inner sep=1pt] (P1O3) at (3.53,-4.85) {};
        \node[draw, circle, fill, inner sep=1pt] (P1O4) at (-3.53,-4.85) {};
        \node[draw, circle, inner sep=1pt] (P1O5) at (-5.71,1.85) {1};
    
        \draw[-,thick] (P1O1) to (P1O2) to (P1O3) to (P1O4) to (P1O5) to (P1O1);
        
        \node[draw, circle, fill, inner sep=1pt] (P1I1) at (0,3) {};
        \node[draw, circle, fill, inner sep=1pt] (P1I2) at (2.85,0.93) {};
        \node[draw, circle, inner sep=1pt] (P1I3) at (1.76,-2.43) {1};
        \node[draw, circle, inner sep=1pt] (P1I4) at (-1.76,-2.43) {1};
        \node[draw, circle, fill, inner sep=1pt] (P1I5) at (-2.85,0.93) {};
    
        \draw[-,thick] (P1I1) to (P1I3) to (P1I5) to (P1I2) to (P1I4) to (P1I1);
    
        \draw[-,thick] (P1O1) to (P1I1);
        \draw[-,thick] (P1O2) to (P1I2);
        \draw[-,thick] (P1O3) to (P1I3);
        \draw[-,thick] (P1O4) to (P1I4);
        \draw[-,thick] (P1O5) to (P1I5);

        \node (S1) at (10,0) {\huge $\sim$};

        \node[draw, circle, fill, inner sep=1pt] (P2O1) at ([shift={(20,0)}]P1O1) {};
        \node[draw, circle, fill, inner sep=1pt] (P2O2) at ([shift={(20,0)}]P1O2) {};
        \node[draw, circle, fill, inner sep=1pt] (P2O3) at ([shift={(20,0)}]P1O3) {};
        \node[draw, circle, fill, inner sep=1pt] (P2O4) at ([shift={(20,0)}]P1O4) {};
        \node[draw, circle, fill, inner sep=1pt] (P2O5) at ([shift={(20,0)}]P1O5) {};
    
        \draw[-,thick] (P2O1) to (P2O2) to (P2O3) to (P2O4) to (P2O5) to (P2O1);
        
        \node[draw, circle, fill, inner sep=1pt] (P2I1) at ([shift={(20,0)}]P1I1) {};
        \node[draw, circle, inner sep=1pt] (P2I2) at ([shift={(20,0)}]P1I2) {2};
        \node[draw, circle, fill, inner sep=1pt] (P2I3) at ([shift={(20,0)}]P1I3) {};
        \node[draw, circle, fill, inner sep=1pt] (P2I4) at ([shift={(20,0)}]P1I4) {};
        \node[draw, circle, inner sep=1pt] (P2I5) at ([shift={(20,0)}]P1I5) {2};
    
        \draw[-,thick] (P2I1) to (P2I3) to (P2I5) to (P2I2) to (P2I4) to (P2I1);
    
        \draw[-,thick] (P2O1) to (P2I1);
        \draw[-,thick] (P2O2) to (P2I2);
        \draw[-,thick] (P2O3) to (P2I3);
        \draw[-,thick] (P2O4) to (P2I4);
        \draw[-,thick] (P2O5) to (P2I5);

        \node (S2) at (30,0) {\huge $\sim$};

        \node[draw, circle, fill, inner sep=1pt] (P3O1) at ([shift={(20,0)}]P2O1) {};
        \node[draw, circle, fill, inner sep=1pt] (P3O2) at ([shift={(20,0)}]P2O2) {};
        \node[draw, circle, inner sep=1pt] (P3O3) at ([shift={(20,0)}]P2O3) {2};
        \node[draw, circle, inner sep=1pt] (P3O4) at ([shift={(20,0)}]P2O4) {2};
        \node[draw, circle, fill, inner sep=1pt] (P3O5) at ([shift={(20,0)}]P2O5) {};
    
        \draw[-,thick] (P3O1) to (P3O2) to (P3O3) to (P3O4) to (P3O5) to (P3O1);
        
        \node[draw, circle, fill, inner sep=1pt] (P3I1) at ([shift={(20,0)}]P2I1) {};
        \node[draw, circle, fill, inner sep=1pt] (P3I2) at ([shift={(20,0)}]P2I2) {};
        \node[draw, circle, fill, inner sep=1pt] (P3I3) at ([shift={(20,0)}]P2I3) {};
        \node[draw, circle, fill, inner sep=1pt] (P3I4) at ([shift={(20,0)}]P2I4) {};
        \node[draw, circle, fill, inner sep=1pt] (P3I5) at ([shift={(20,0)}]P2I5) {};
    
        \draw[-,thick] (P3I1) to (P3I3) to (P3I5) to (P3I2) to (P3I4) to (P3I1);
    
        \draw[-,thick] (P3O1) to (P3I1);
        \draw[-,thick] (P3O2) to (P3I2);
        \draw[-,thick] (P3O3) to (P3I3);
        \draw[-,thick] (P3O4) to (P3I4);
        \draw[-,thick] (P3O5) to (P3I5);

        \node (S2) at (50,0) {\huge $\sim$};

        \node[draw, circle, inner sep=1pt] (P4O1) at ([shift={(20,0)}]P3O1) {2};
        \node[draw, circle, fill, inner sep=1pt] (P4O2) at ([shift={(20,0)}]P3O2) {};
        \node[draw, circle, fill, inner sep=1pt] (P4O3) at ([shift={(20,0)}]P3O3) {};
        \node[draw, circle, fill, inner sep=1pt] (P4O4) at ([shift={(20,0)}]P3O4) {};
        \node[draw, circle, fill, inner sep=1pt] (P4O5) at ([shift={(20,0)}]P3O5) {};
    
        \draw[-,thick] (P4O1) to (P4O2) to (P4O3) to (P4O4) to (P4O5) to (P4O1);
        
        \node[draw, circle, inner sep=1pt] (P4I1) at ([shift={(20,0)}]P3I1) {2};
        \node[draw, circle, fill, inner sep=1pt] (P4I2) at ([shift={(20,0)}]P3I2) {};
        \node[draw, circle, fill, inner sep=1pt] (P4I3) at ([shift={(20,0)}]P3I3) {};
        \node[draw, circle, fill, inner sep=1pt] (P4I4) at ([shift={(20,0)}]P3I4) {};
        \node[draw, circle, fill, inner sep=1pt] (P4I5) at ([shift={(20,0)}]P3I5) {};
    
        \draw[-,thick] (P4I1) to (P4I3) to (P4I5) to (P4I2) to (P4I4) to (P4I1);
    
        \draw[-,thick] (P4O1) to (P4I1);
        \draw[-,thick] (P4O2) to (P4I2);
        \draw[-,thick] (P4O3) to (P4I3);
        \draw[-,thick] (P4O4) to (P4I4);
        \draw[-,thick] (P4O5) to (P4I5);
    \end{tikzpicture}
    \caption{A winning degree $4$ divisor on $\KG(5,2)$.}
    \label{fig:petersen degree 4 example}
\end{figure}
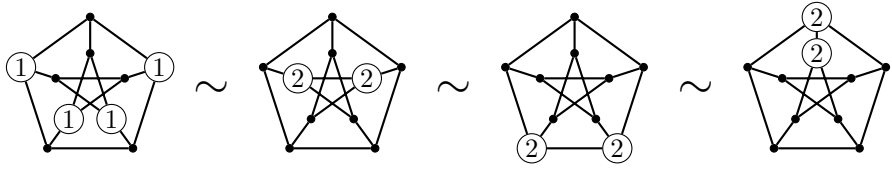

More generally, there is an infinite set of Kneser graphs (namely the ($k+1$)-odd graphs) for which the gonality is lower than what is given in \cref{lem:gon < n - alpha for simple graphs}. That is, while \cref{lem:gon < n - alpha for simple graphs} is a strong upper bound for gonality, it is certainly not always tight.

\begin{restatable}{proposition}{oddgraphgonality}
\label{prop: odd graph gonality}
   Let $k\geq 2$ and $n=2k+1$. Then $$\gon(\KG(n,k))\leq \ind(\KG(n,k)) = \binom{2k}{k-1} < \binom{2k}{k} = |V(\KG(n,k))|-\ind(\KG(n,k)).$$
\end{restatable}
\begin{proof}
    Let $k \geq 2$ and $n = 2k+1$.
    Then $\ind(\KG(n,k)) = \binom{2k}{k-1}$ by \cref{lem:n >= 2k independence}, and $|V(\KG(n,k))| - \ind(\KG(n,k)) = \binom{2k}{k}$, as in \cref{lem:gon < n - alpha for simple graphs}.
    Therefore, as $\binom{2k}{k} > \binom{2k}{k-1}$ for all $k > 1$, we have the right side of the desired inequality.
    
    To demonstrate that $\gon(\KG(n,k)) \leq \ind(\KG(n,k))$, we seek a winning divisor with degree $\ind(\KG(n,k))$.
    Let $\indset = \{A \in \binom{[2k+1]}{k} : 2k+1 \in A\}$. 
    
    Using $v_A$ to denote the vertex in $\KG(2k+1,k)$ corresponding to the set $A \in \binom{[2k+1]}{k}$, we show that the divisor $D = \sum_{A \in \indset}v_A$ 
     is a winning divisor.

    Consider an arbitrary vertex $v_X$.
    If $X \in \indset$, then $D - v_X$ is an effective divisor and we are done.
    Otherwise, if $X \notin \indset$, then $2k+1 \notin X$ and $X$ is a $k$-subset of $[2k]$. 
    Then there is a unique vertex $v_Y$ connected to $v_X$ such that $Y \notin \indset$, corresponding to $Y = [2k] \setminus X$.
    Therefore, any other vertex $v_Z \neq v_Y$ adjacent to $v_X$ has $Z \in \indset$, and firing all vertices in $V(\KG(2k+1,k))$ except $v_X$ and $v_Y$ transfers chips only from vertices in $\indset$ to $v_X$ and $v_Y$.
    By \cref{lem:kneser clique number} $\KG(2k+1,k)$ is triangle-free and no vertex connects to both $v_X$ and $v_Y$, so every vertex loses at most one chip in the simultaneous chip firing move. 
    Therefore, because each vertex $v_A$ with $A \in \indset$ begins with a chip, the divisor resulting from this chip firing move is effective, and $D - v_X$ is always equivalent to an effective divisor.

    As $D - v_X$ is always equivalent to an effective divisor, we have $\rk(D) \geq 1$, and thus $$ \gon(\KG(2k+1,k)) \leq \deg(D) = |\indset| = \alpha(\KG(2k+1,k)).$$
\end{proof}

We now turn to techniques for lower bounding gonality.

\subsection{Lower Bounds}
\label{ssec: scramble num}

In contrast to the previous \lcnamecref{ssec: upper bounds for gonality}, determining a lower bound for gonality is difficult \cite{ECHAVARRIA202243,gijswijt2020computing,voigt2016tree}.
We begin by presenting a classically studied lower bound for gonality to provide context for our improvement, and then introduce the invariant we will be exploiting.

\emph{Treewidth}, denoted $\tw$, is a graph invariant introduced by Robertson and Seymour \cite{ROBERTSON} that was recently found to be a lower bound for gonality \cite{van_Dobben_de_Bruyn_2020}. 
Treewidth has been studied for Kneser graphs by Harvey and Wood in \cite{harvey2013treewidthknesergrapherdhoskorado} and Liu, Cao, and Lu in \cite{liu2021treewidthgeneralizedknesergraphs}. Both present the following \lcnamecref{lem:tw poly bound}, which yields \cref{cor:tw poly bound gon}, as treewidth is a lower bound for gonality.

\begin{theorem}[{{\cite{harvey2013treewidthknesergrapherdhoskorado}*{Theorem 1}}}]
\label{lem:tw poly bound}
    Let $k \geq 3$ and $n\geq 4k^2-3k+2$. Then
    $$\tw(\KG(n,k)) =\binom{n-1}{k}-1.$$
\end{theorem}
 
\begin{corollary}
\label{cor:tw poly bound gon} 
    Let $k\geq 3$ and $n\geq4k^2-3k+2$. Then $$\gon(\KG(n,k)) \geq \binom{n-1}{k}-1.$$
\end{corollary}

We find special interest in \cref{cor:tw poly bound gon}, as it 
restricts $\gon(\KG(n,k))$ to either $\tw(\KG(n,k))$ or $\tw(\KG(n,k))+1$ when $n \geq 4k^2 - 3k + 2$ when combined with \cref{lem:gon < n - alpha for simple graphs} (see \cref{cor: bounds on KG}).
Our main result shows that $\gon(\KG(n,k)) = \tw(\KG(n,k)) + 1$ when $n \geq 4k^2 - 3k + 2$, and that $\gon(\KG(n,k)) = \binom{n-1}{k}$ for more values of $n$ and $k$ than the treewidth has been computed for.
We do this using the invariant scramble number, a tighter lower bound for the gonality than treewidth.

Harp, Jackson, Jensen, and Speeter introduced scramble number in \cite{HJJS} to provide a better lower bound for gonality than treewidth.
To define this invariant, we provide a series of auxiliary definitions, 
where $\graph$ represents any graph and we use $E(\graph)$ to represent the set of edges in $\graph$.

\begin{definition}
    An \emph{egg} $\mathcal{E} \subset V(\graph)$ is a nonempty connected set of vertices of $\graph$. That is, the vertices in $\mathcal{E}$ induce a nonempty connected subgraph.
    A \emph{scramble} $\scramble$ on $\graph$ is a finite set of eggs. 
\end{definition}

Scramble number is defined in terms of two criteria of sets of vertices and edges, respectively: ``hitting" and ``cutting." 
We say that a set $A$ of vertices \emph{hits} an egg $\mathcal{E}$ if at least one of the vertices in $A$ is in $\mathcal{E}$. 
If $A$ hits all of the eggs in a given scramble, we call it a \emph{hitting set} of that scramble. The minimum size of a hitting set for a scramble $\scramble$ is called the \emph{hitting size}, denoted $h(\scramble)$.
In contrast, a set $C$ of edges \emph{cuts} the graph if $\graph$ is disconnected when all of the edges in $C$ are removed.
We call $C$ an \emph{egg-cut} if at least one egg remains in each of at least two connected components of the resulting graph (that is, at least two eggs do not contain edges in $C$, and end up in distinct components). The minimum size of an egg cut for a scramble $\scramble$ is called the \emph{cut size}, denoted $e(\scramble)$. If no such egg cut exists, then we take $e(\scramble) = \infty$.

For an example, see Figure \ref{fig:petersen scramble example} where a scramble with $5$ disjoint eggs is illustrated on the Petersen graph. Since all $5$ eggs are disjoint we need at least $5$ vertices to form a hitting set. We may produce a hitting set of $5$ vertices by selecting the vertices on the outside cycle, showing the hitting size is exactly $5$. To produce a cut, we note that between each pair of eggs, there are four disjoint paths. This implies that we require at least $4$ edges to disconnect an egg. We may then produce an egg cut by isolating a specific egg. That is, take any egg in Figure \ref{fig:petersen scramble example} and produce a cut by removing the $4$ edges that leave the egg.

One might have some intuition as to how these relate to gonality, whereby eggs create smaller subgraphs of interest which we need to either already have chips on or be able to transfer chips to. Hitting size gives the smallest number of chips necessary to place a chip on each of those subgraphs, and cut size relates to the smallest number of chips required to transfer from one of these smaller graphs to another without creating any debt.
Combining these characteristics, we define the order of a scramble $\scramble$.

\begin{definition}

    The \emph{order} of a scramble $\scramble$ is $\ord{\scramble} = \min\{h(\scramble), e(\scramble)\}$. 
\end{definition}

Finally, we define the scramble number of a graph $\graph$.

\begin{definition}
    The \emph{scramble number} $\sn(\graph)$ of a graph $\graph$ is $$\sn(\graph) = \max\{\ord{\scramble} : \scramble \text{ is a scramble on } \graph\}.$$
\end{definition}

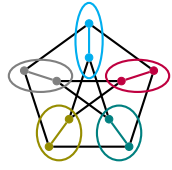
\begin{figure}[h]
    \centering
    \begin{tikzpicture}[scale=0.15]
        \node[draw, circle, fill, cyan, inner sep=1pt] (A) at (0,6) {};
        \node[draw, circle, fill, purple, inner sep=1pt] (B) at (5.7,1.85) {};
        \node[draw, circle, fill, teal, inner sep=1pt] (C) at (3.53,-4.85) {};
        \node[draw, circle, fill, olive, inner sep=1pt](D) at (-3.53,-4.85) {};
        \node[draw, circle, fill, gray, inner sep=1pt](E) at (-5.71,1.85) {};
        
        \draw[-,thick] (A) to (B);
        \draw[-,thick] (B) to (C);
        \draw[-,thick] (C) to (D);
        \draw[-,thick] (D) to (E);
        \draw[-,thick] (A) to (E);
        
        \node[draw, circle, fill, cyan, inner sep=1pt] (F) at (0,3) {};
        \node[draw, circle, fill, purple, inner sep=1pt] (G) at (2.85,0.93) {};
        \node[draw, circle, fill, teal, inner sep=1pt] (H) at (1.76,-2.43) {};
        \node[draw, circle, fill, olive, inner sep=1pt](I) at (-1.76,-2.43) {};
        \node[draw, circle, fill, gray, inner sep=1pt](J) at (-2.85,0.93) {};
        
        \draw[-,thick] (F) to (H);
        \draw[-,thick] (G) to (I);
        \draw[-,thick] (H) to (J);
        \draw[-,thick] (I) to (F);
        \draw[-,thick] (J) to (G);
        
        \draw[-,thick,cyan] (A) to (F);
        \draw[-,thick,purple] (B) to (G);
        \draw[-,thick,teal] (C) to (H);
        \draw[-,thick,olive] (D) to (I);
        \draw[-,thick,gray] (E) to (J);

        \node[draw=cyan, thick, ellipse, inner sep=2pt, fit=(A) (F)]{};
        \node[draw=purple, thick, ellipse, inner sep=.7pt, fit=(B) (G)]{};
        \node[draw=teal, thick, ellipse, inner sep=.5pt, fit=(C) (H)]{};
        \node[draw=olive, thick, ellipse, inner sep=.5pt, fit=(D) (I)]{};
        \node[draw=gray, thick, ellipse, inner sep=.7pt, fit=(E) (J)]{};
    \end{tikzpicture}
    \caption{Scramble $\scramble$ with $h(\scramble) = 5$ and $e(\scramble) = 4$.}
    \label{fig:petersen scramble example}
\end{figure}

We focus on scramble number as it is a tighter bound for gonality than treewidth, as shown in \cite{HJJS}.
\begin{lemma}[\cite{HJJS}*{Theorem 1.1}]
\label{lem:sn <= gon}
    For any graph $\graph$, $$\tw(\graph)\leq\sn(\graph) \leq \gon(\graph).$$
\end{lemma}

\Cref{fig:petersen degree 4 example} exhibits a winning divisor on $\KG(5,2)$ of degree 4. 
\Cref{fig:petersen scramble example} demonstrates that a scramble of order $4$ on $\KG(5,2)$ exists. 
Altogether, this suffices to show that $\sn(\KG(5,2)) = \gon(\KG(5,2)) = 4$, as $$4\leq \sn(\KG(5,2)) \leq \gon(\KG(5,2)) \leq 4.$$

\cref{{lem:sn <= gon}}, together with \cref{lem:gon < n - alpha for simple graphs}, implies the following
\begin{corollary}
    \label{cor: bounds on KG}
    Let $n>2k$. Then $$\sn(\KG(n,k))\leq \gon(\KG(n,k))\leq \binom{n-1}{k}.$$
\end{corollary}
Combining this result with \cref{lem:tw poly bound} and \cref{cor:tw poly bound gon}, we have that when $n\geq 4k^2-3k+2$, $\sn(\KG(n,k))$ takes one of two values: $\binom{n-1}{k}-1$ or $\binom{n-1}{k}$. 

In Section \ref{sec: main work}, we prove that scramble number takes the latter value for $n\geq \frac{3k^2+k+2}{2}$ by applying the following lemma from \cite{ECHAVARRIA202243} to Kneser graphs.

\begin{lemma}[\cite{ECHAVARRIA202243}*{Corollary 3.2}]
\label{lem:sn = gon for simple graphs}
    For a simple graph $\graph$, if $\dgr(\graph) \geq \floor{\frac{|V(\graph)|}{2}} + 1$, then $$\sn(\graph) = \gon(\graph) = |V(\graph)| - \ind(\graph),$$ where $\dgr(\graph)$ denotes the minimum degree of any vertex in $\graph$.
\end{lemma} 

Applying this \lcnamecref{lem:sn = gon for simple graphs} to $\KG(n,k)$ for $n>2k$, we have the following result.

\begin{corollary}
\label{lem: n > 2k then sn = gon bound}
    Let $n >2k$ and $\binom{n-k}{k} > \frac{1}{2}\binom{n}{k}$. Then $$\sn(\KG(n,k)) = \gon(\KG(n,k)) = \binom{n-1}{k}.$$
\end{corollary}

\begin{proof}
    Suppose that $\binom{n-k}{k} > \frac{1}{2}\binom{n}{k}$ and $n > 2k$.
    As $\binom{n-k}{k} > \frac{1}{2}\binom{n}{k} \geq \floor{\frac{1}{2}\binom{n}{k}}$, we have $\binom{n-k}{k} \geq \floor{\frac{1}{2}\binom{n}{k}} + 1$.
    Because $|V(\KG(n,k))| = \binom{n}{k}$ and all vertices in $\KG(n,k)$ have degree $\binom{n-k}{k}$, we have $$\sn(\KG(n,k)) = \gon(\KG(n,k)) = |V(\KG(n,k))| - \alpha(\KG(n,k)) = \binom{n-1}{k}$$ by \cref{lem:sn = gon for simple graphs,lem:n >= 2k independence}.
\end{proof}

That is, an infinite set of Kneser graphs have gonality exactly equal to $\binom{n-1}{k}$ (one more than the treewidth). 
We dedicate the next \lcnamecref{sec: main work} to characterizing this subset with a bound $n \geq f(k)$, where $f$ is a polynomial in $k$.

\section{Exploring a Better Polynomial Bound}
\label{sec: main work}

Recalling the polynomial bound given in  \cref{lem:tw poly bound},
we have that for $n \geq 4k^2 - 3k + 2$, $\gon(\KG(n,k)) \geq \binom{n-1}{k}-1$. In this section, we show that for $n \geq \frac{3k^2 + k + 2}{2}$, we have $\gon(\KG(n,k)) = \binom{n-1}{k}$, a polynomial which determines gonality of $\KG(n,k)$ for more values of $n$ than the previously established polynomial for all $k \in \NN$. A comparison of these bounds is provided in \cref{fig:polynomial comparison}. 

To prove this result, we use \cref{lem: n > 2k then sn = gon bound}, first showing that if $\binom{n-k}{k} > \frac{1}{2}\binom{n}{k}$ for a fixed value $n = N$ and fixed $k$, then the same inequality holds for all $n \geq N$ with the same $k$, a fact which allows us to focus explicitly on the minimum value of $n$ such that the inequality holds for a fixed $k$. To complete the proof, we show that $\binom{n-k}{k}>\frac{1}{2}\binom{n}{k}$ holds for $N = \frac{3k^2 + k + 2}{2}$ for all $k \in \NN$ (\cref{fig:polynomial inequality comparison}), making this polynomial a lower bound for the gonality result given above.

\begin{figure}[h]\centering
    \includegraphics[]{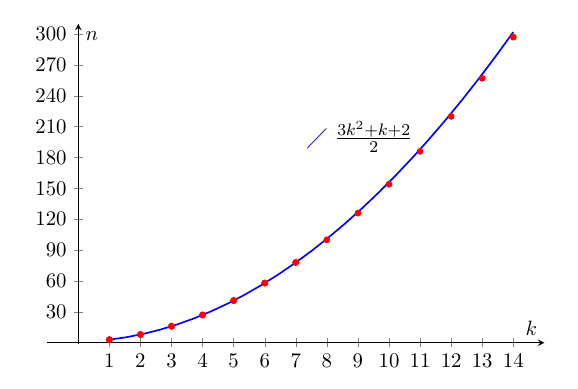}
    \caption{Comparison of polynomial bound (blue) with minimum values of $n$ such that $\binom{n-k}{k} > \frac{1}{2}\binom{n}{k}$ (red).}
    \label{fig:polynomial inequality comparison}
\end{figure}

With the goal of proving that when $n\geq\frac{3k^2+k+2}{2}$, $\binom{n-k}{k}>\frac{1}{2}\binom{n}{k}$, we use calculus techniques. We need to show two things:

\begin{enumerate}
    \item[(1)] $$\frac{\binom{n-k}{k}}{\binom{n}{k}} \hspace{1em} \text{is increasing on $n$ with a fixed $k$, and}$$
    \item[(2)] $$\frac{\binom{n-k}{k}}{\binom{n}{k}}>\frac{1}{2} \hspace{1em} \text{when $n = f(k)$ and $k\geq 2$}.$$
\end{enumerate} 

The following section details the calculus we needed to prove this, as well as other lemmas used in the proof of our main theorem.

\subsection{Required Calculus}
\label{sec: calc}
Throughout this section, we study the behavior of $\displaystyle \frac{\binom{n-k}{k}}{\binom{n}{k}}$. 
Recall that binomial coefficients can be expressed as a quotient of products of factorial functions; because we apply calculus techniques to functions of this form, we must take a smooth extension of these factorial functions, given by the Gamma function.
In this section, whenever we perform calculus involving factorials, we implicitly refer to its smooth extension, whereby showing monotonicity for the smooth extension also proves it for the discrete version as well. We then prove a lemma and corollary that we use in the differentiation of these factorial functions, which applied to the proofs of \Cref{thm: 3k^2+k lower bound} and \Cref{prop: conj poly bound} reduces the necessary algebra significantly.

Using algebraic manipulation, we rewrite
$$ \frac{\binom{n-k}{k}}{\binom{n}{k}} = \prod_{i=0}^{k-1}\frac{n-k-i}{n-i},$$
a representation which we rely on extensively in this section.

As a result of this, we apply previously known results from Section \ref{ssec: scramble num} to this product representation in the following manner.

\begin{remark}
    Recall $\KG(n,1)$ is the complete graph $K_n$.
    Evaluating $\prod_{i=0}^{k-1}\frac{n-k-i}{n-i}$ for $k = 1$, we obtain $\frac{\binom{n-k}{k}}{\binom{n}{k}} = 1 - \frac{1}{n} > \frac{1}{2}$ for all $n\geq 3>2k$. Thus, $\gon(\KG(n,1)) = \binom{n-1}{1} = n-1$.
    With \Cref{lem: n > 2k then sn = gon bound}, 
    we recover the known result that for complete graphs, $\sn(K_n) = \gon(K_n) = n - 1$ \cite{ECHAVARRIA202243}.
\end{remark}

We now demonstrate that for a fixed $k$, $\frac{\binom{n-k}{k}}{\binom{n}{k}}$ is strictly increasing in $n$.

\begin{lemma}
\label{lem:characteristics of f(n)}
    Let $f(x) = 2\prod_{i=0}^{k-1}\frac{x-k-i}{x-i}$ for some fixed positive integer $k$. Then if $k \geq 2$, $\{f(n)\}_{n=1}^{\infty}$ is strictly increasing whenever $n > 2k$.
\end{lemma}

\begin{proof}
    We prove that the sequence $\{f(n)\}_{n=1}^{\infty}$ is strictly increasing by showing that $\frac{d{f}}{dx}$ is postive and then restricting to $x\in \NN$.
    
    Note that $f$ has a defined derivative only if it is continuous. Let us verify that the derivative of $f$ is defined on $[k, \infty)$. Using generalized product rule,
    \begin{align*}
        \frac{d{f}}{dx}
        &= 2\frac{d}{dx}\left(\prod_{i=0}^{k-1}\frac{x-k-i}{x-i}\right) \\
        &= 2\left(\prod_{i=0}^{k-1}\frac{x-k-i}{x-i}\right)\left(\sum_{i=0}^{k-1}\frac{\frac{d}{dx}\left(\frac{x-k-i}{x-i}\right)}{\frac{x-k-i}{x-i}}\right) \\
        &= 2\left(\prod_{i=0}^{k-1}\frac{x-k-i}{x-i}\right)\left(\sum_{i=0}^{k-1}\left(\left(\frac{(x-i) - (x-k-i)}{(x-i)^2}\right) \cdot \frac{x-i}{x-k-i}\right)\right) \\
        &= 2\left(\prod_{i=0}^{k-1}\frac{x-k-i}{x-i}\right)\left(\sum_{i=0}^{k-1}\left(\frac{k}{x-i} \cdot \frac{1}{x-k-i}\right)\right) \\
        &= 2k\left(\prod_{i=0}^{k-1}\frac{x-k-i}{x-i}\right)\left(\sum_{i=0}^{k-1}\frac{1}{(x-i)(x-k-i)}\right).
    \end{align*}
    This implies that $\frac{d{f}}{dx}$ is undefined only when any of its denominator terms is equal to zero. 
    Because the $x-k-i$ term in the numerator of the product cancels with any instances of $x-k-i$, $0 \leq i \leq k-1$, in the denominator of the sum, we have that 
    $\frac{df}{dx}$ is undefined only at $x = i$ for $0 \leq i \leq k-1$.
    As such, $\frac{d{f}}{dx}$ is defined on $(2k, \infty) \subseteq (k,\infty)$ and so $f$ is differentiable, and thus also continuous, on this interval. 

    Finally, to show that $f$ is strictly increasing on $(2k, \infty)$, we demonstrate that $\frac{d{f}}{dx}$ is positive on this interval.
    Considering the multiplicands of $\frac{d{f}}{dx}$, we have $2k > 0$, as $k \geq 2 > 0$, and when $x > 2k$, we have $x-i > x-k-i > 0$ for $0 \leq i \leq k$, so all multiplicands of $\frac{d{f}}{dx}$ are positive.

    Therefore, $\frac{d{f}}{dx}$ is positive on $(2k, \infty)$, so $f$, and thus the sequence $\{f(n)\}_{n=1}^{\infty}$, is increasing on the interval. 
\end{proof}

Notably, for the given $k$, $f(k) = 0$ and $\lim_{x \rightarrow \infty}f(x) = 2$, so by Intermediate Value Theorem, there must be some $x > k$ such that $f(x) > 1$. As a consequence, restricting the domain of $f$ to $\NN$, for a given $k$, there must be an $N \in \mathbb{N}$ such that $\sn(\KG(n,k)) = \gon(\KG(n,k)) = \binom{n-1}{k}$ for all $n \geq N$, as noted at the beginning of this section. In \cref{thm: 3k^2+k lower bound}, we give a formula for such an $N$ in terms of $k$.

Next, we show that when $n\geq \frac{3k^2+k+2}{2}$, we have $\frac{\binom{n-k}{k}}{\binom{n}{k}}>\frac{1}{2}$ for all $k\geq 2$. By our results from before, it suffices to show that it is true for $n=\frac{3k^2+k+2}{2}$, because our function is increasing in $n$. Once again we employ calculus techniques.
Since factorial functions are only defined at non-negative integer values, we make use of the technique of taking the derivative of a smooth extension of the factorial function via the Gamma function, and analyzing its behavior at non-negative integer values. 
To this end, for a differentiable function $f:\RR \to \RR$ such that $f(\NN)\subseteq \NN$ and $f(x)+1>0$ for all $x\in \RR$, we define $f^!(x):=\Gamma(f(x)+1)$ to be \emph{the smooth extension of $(f(x))!$} under the Gamma function. Note that given these conditions, $f^!(x)$ is differentiable on $\RR$.

It is well known that for a positive integer $m$, the derivative of the Gamma function can be calculated as $$\Gamma'(m+1)=m!\left(-\gamma+\sum_{i=1}^m\frac{1}{i}\right),$$
where $\sum_{i=1}^m\frac{1}{i}$ is the $m$th harmonic number and $\gamma$ is the Euler-Mascheroni constant \cite{abramowitz1972psi}.
In particular, for positive integer values $k$ we have that $f^!(k)$ only outputs positive integer values. Thus, we are able to apply the above calculation in taking the derivative of $f^!(x)$ when $x$ is restricted to $k\in \NN$. 

Notice that the expression $\frac{\binom{n-x}{x}}{\binom{n}{x}}$ can be rewritten as $\frac{f_1^!(x)\cdot f_2^!(x)}{f_3^!(x)\cdot f_4^!(x)}$, where each $f_i^!(x)$ is a function of $x$ dependent on the choice of $n$. We prove this function is increasing for a fixed $n=f(x)$ when $x$ is restricted to $k\in \NN$ for $k>a$  by showing that the derivative is positive for all $k>a$, for some $a\in \NN$. By \Cref{lem:characteristics of f(n)}, we know that $\frac{\binom{n-k}{k}}{\binom{n}{k}}$ is increasing in $n$ for a fixed $k$. This allows us to take the natural logarithm of $\frac{\binom{n-k}{k}}{\binom{n}{k}}$, as the natural logarithm of an increasing function is itself increasing. Therefore, we can examine $\ln\left(\frac{\binom{n-x}{k}}{\binom{n}{x}}\right)$ when $x$ is restricted to $k\in\NN$ as a sum and take the derivatives of each summand individually.

Doing so allows us to look at the derivative of each $f_i^!(x)$ in terms of the product of the derivative of Gamma function extension of $(f_i(x))!$ and $f_i'(x)$, given by the following lemma.
\begin{lemma}
\label{lem: derivative of ln(f(x)!)}
    Let $f:\RR\to \RR$ be a differentiable function such that $f(\NN)\subseteq \NN$ and $f(x)+1>0$ for all $x\in \RR$, and let $f^!$ be the smooth extension of $f$ under the Gamma function. Then the derivative of $\ln(f^!(x))$ when $x$ is restricted to $k\in\NN$ is 
    \begin{align*}
        \frac{d}{dx}\ln(f^!(x))\bigg|_{x=k}=\left(-\gamma+\sum_{i=1}^{f(k)}\frac{1}{i}\right)\cdot f'(k)
    \end{align*}
\end{lemma}
\begin{proof}
    Consider $\ln(f^!(x))$ for $x\in \RR$. Since $f^!(x)>0$ and is differentiable by the properties of the Gamma function, we have that \begin{align*}
        \frac{d}{dx}\ln(f^!(x))&=\frac{\frac{d}{dx}f^!(x)}{f^!(x)}\\
        &=\frac{\frac{d}{dx}\Gamma(f(x)+1)}{\Gamma(f(x)+1)}.
    \end{align*}
    
    Restricting to $k\in \NN$, we have using the known derivative of the Gamma function for positive integers that \begin{align*}
        \frac{d}{dx}\ln(f^!(x))\bigg|_{x=k}&=\frac{\frac{d}{dk}\Gamma(f(k)+1)}{\Gamma(f(k)+1)}\\
        &=\frac{(f(k))!(-\gamma+\sum_{i=1}^{f(k)}\frac{1}{i})}{(f(k))!}\cdot f'(k)\\
        &=\left(-\gamma+\sum_{i=1}^{f(k)}\frac{1}{i}\right)\cdot f'(k). \qedhere
    \end{align*}
\end{proof}

We remark that for the purposes of our work, we do not need to be concerned with the exact value of $\gamma$. Specifically, when the difference of the derivatives of the functions in the numerator and the derivatives of the functions in the denominator of $\frac{\binom{n-k}{k}}{\binom{n}{k}}$ is zero, we have that the $\gamma$ terms will annihilate when calculating the derivatives of $\ln\left(\frac{\binom{n-x}{x}}{\binom{n}{x}}\right)$ when $x$ is restricted to $k\in \NN$, as shown in the following corollary.

\begin{corollary}
\label{cor:macaroni annihilation}
    Let $\ell, m \in \NN$, let $f_i,g_j:\RR\to \RR$ be differentiable functions for $i\in [\ell]$ and $j\in [m]$ such that $f_i(\NN)\subseteq \NN, g_j(\NN)\subseteq \NN$, $f_i(x)+1>0$, and $g_j(x)+1>0$ for all $x\in \RR$, and let $f_i^!$ and $g_j^!$ be the smooth extensions of $(f_i(x))!$ and $(g_j(x))!$ under the Gamma function, respectively. If $\sum_{i=1}^\ell f_i'(x)-\sum_{j=1}^mg_j'(x)=0$, we have $$\frac{d}{dx}\ln\left(\frac{\prod_{i=1}^\ell f_i^!(x)}{\prod_{j=1}^m g_j^!(x)}\right)\bigg|_{x=k\in\NN} = \sum_{i=1}^\ell\left(f_i'(k)\cdot\sum_{n=1}^{f_i(k)}\frac{1}{n}\right) - \sum_{j=1}^m\left(g_j'(k)\cdot\sum_{n=1}^{g_j(k)}\frac{1}{n}\right).$$

\end{corollary}
\begin{proof} 
    We have
    \begin{align*}
        \ln\left(\frac{\prod_{i=1}^\ell f_i(x)!}{ \prod_{j=1}^m g_j(x)!}\right) &= \ln\left(\prod_{i=1}^\ell f_i(x)!\right) - \ln\left(\prod_{j=1}^m g_j(x)!\right) 
        = \sum_{i = 1}^\ell \ln(f_i(x)!) - \sum_{i = 1}^m\ln(g_j(x)!).
    \end{align*}
    Taking the derivative, we get that
    \begin{align*}
        \frac{d}{dx}\ln(\frac{\prod_{i=1}^\ell f_i^!(x)}{\prod_{j=1}^m g_j^!(x)}) &= \frac{d}{dx}\sum_{i = 1}^\ell \ln(f_i^!(x)) - \frac{d}{dx}\sum_{j = 1}^ m\ln(g_j^!(x)) \\
        &= \sum_{i = 1}^\ell\frac{d}{dx} \ln(f_i^!(x)) - \sum_{j = 1}^ m\frac{d}{dx}\ln(g_j^!(x)). 
    \end{align*}
    Restricting $x$ to $k\in\NN$ and letting $S(f_i) = \sum_{n=1}^{f_i(k)}\frac{1}{n}$ and $S(g_j)=\sum_{n=1}^{g_j(k)} \frac{1}{n}$, we have by \Cref{lem: derivative of ln(f(x)!)} that
    \begin{align*}
       \frac{d}{dx}\ln\left(\frac{\prod_{i=1}^\ell f_i^!(x)}{\prod_{j=1}^m g_j^!(x)}\right)\bigg|_{x=k\in\NN} &= \sum_{i = 1}^\ell \Big(\big(-\gamma + S(f_i)\big)\cdot f_i'(k)\Big) - \sum_{j = 1}^m\Big(\big(-\gamma + S(g_j)\big)\cdot g_j'(k)\Big) \\
        &= -\gamma\cdot\sum_{i = 1}^\ell f_i'(k) + \sum_{i = 1}^\ell (S(f_i)\cdot f_i'(k)) - \left(-\gamma\cdot\sum_{j = 1}^m g_j'(k) + \sum_{j = 1}^m(S(g_j)\cdot g_j'(k))\right) \\
        &= -\gamma\left(\sum_{i = 1}^\ell f_i'(k) - \sum_{j = 1}^m g_j'(k)\right) + \sum_{i = 1}^\ell (S(f_i)\cdot f_i'(k)) - \sum_{j = 1}^m(S(g_j)\cdot g_j'(k))
    \end{align*}
    
    Since $\sum_{i=1}^\ell f_i'(x) - \sum_{j=1}^m g_j'(x) = 0$ for all $x\in\RR$, we have 
    \begin{equation*}
        \frac{d}{dx}\ln(\frac{\prod_{i=1}^\ell f_i^!(x)}{\prod_{j=1}^m g_j^!(x)})\bigg|_{x=k\in\NN} = \sum_{i = 1}^\ell (S(f_i)\cdot f_i'(k)) - \sum_{j = 1}^m(S(g_j)\cdot g_j'(k)).\qedhere
    \end{equation*}
\end{proof}

Equipped with all the necessary calculus techniques, we are ready prove our main result.

\subsection{Proof of Main Result}
\label{sec: proof}

For a fixed $k$, we have by \Cref{lem:characteristics of f(n)} that $\frac{{n-k\choose k}}{{n\choose k}}$ is increasing in $n$.
That is, if $\binom{n-k}{k} > \frac{1}{2}\binom{n}{k}$ holds for $n = N$, then the same inequality holds for all $n \geq N$.
We will show that when $n = \frac{3x^2+x+2}{2}$, $\frac{\binom{n-x}{x}}{\binom{n}{x}}>\frac{1}{2}$ holds when $x=k\in \NN$ for $k\geq 1$, so that by \cref{lem: n > 2k then sn = gon bound}, $\sn(\KG(n,k)) = \gon(\KG(n,k)) = \binom{n-1}{k}$ for all $n \geq \frac{3k^2 + k + 2}{2}$.
\begin{proof}[Proof of \cref{thm: 3k^2+k lower bound}]
    Fix $n = \frac{3x^2 + x + 2}{2}$ and let $g(x)= \frac{\binom{n-x}{x}}{\binom{n}{x}}$. Rewriting, then, $$g(x) = \frac{{\frac{3x^2-x+2}{2}\choose x}}{{\frac{3k^2+x+2}{2}\choose x}} = \frac{\left(\frac{3x^2-x+2}{2}\right)!}{\left(\frac{3x^2-3x+2}{2}\right)!x!}\cdot\frac{\left(\frac{3x^2-x+2}{2}\right)!x!}{\left(\frac{3x^2+x+2}{2}\right)!} = \frac{\left(\frac{3x^2-x+2}{2}\right)!}{\left(\frac{3x^2-3x+2}{2}\right)!}\cdot\frac{\left(\frac{3x^2-x+2}{2}\right)!}{\left(\frac{3x^2+x+2}{2}\right)!}.$$

    Since $g(x)$ satisfies the necessary conditions, we restrict $x$ to $k\in \NN$ and apply \Cref{cor:macaroni annihilation} to obtain $$\frac{d}{dx}\ln(g(x))\bigg|_{x=k}=2\cdot\left(\frac{6k-1}{2}\right)\left(\sum_{i=1}^{\frac{3k^2-k+2}{2}}\frac{1}{i}\right) -\left(\frac{6k-3}{2}\right)\left(\sum_{i=1}^{\frac{3k^2-3k+2}{2}}\frac{1}{i}\right)-\left(\frac{6k+1}{2}\right)\left(\sum_{i=1}^{\frac{3k^2+k+2}{2}}\frac{1}{i}\right).$$
    Extracting a factor of $\frac{6k - 3}{2}$ from each term and grouping these yields
    \begin{align*}
        \frac{d}{dx}\ln(g(x))\bigg|_{x=k} &= 2 \left(\sum_{i=1}^{\frac{3k^2-k+2}{2}}\frac{1}{i}-\sum_{i=1}^{\frac{3k^2+k+2}{2}}\frac{1}{i}\right) +\left(\frac{6k-3}{2}\right)\left( 2\cdot \sum_{i=1}^{\frac{3k^2-k+2}{2}}\frac{1}{i}-\sum_{i=1}^{\frac{3k^2-3k+2}{2}}\frac{1}{i}-\sum_{i=1}^{\frac{3k^2+k+2}{2}}\frac{1}{i} \right) \\
        &= -2\left(\sum_{i=\frac{3k^2-k+2}{2}+1}^{\frac{3k^2+k+2}{2}}\frac{1}{i}\right)+ \left(\frac{6k-3}{2}\right)\left( \sum_{i=\frac{3k^2-3k+2}{2} + 1}^{\frac{3k^2-k+2}{2}}\frac{1}{i}-\sum_{i=\frac{3k^2-k+2}{2}+1}^{\frac{3k^2+k+2}{2}}\frac{1}{i}\right)
    \end{align*}
    We now seek to simplify these summations into a single summation, and then demonstrate that the single summation is positive, in a similar way to the proof of \cref{lem:characteristics of f(n)}.
    Noting that the difference between the upper and lower indices of summation is the same for each summation; that is, $\frac{3k^2 + k + 2}{2} - \left(\frac{3k^2 - k + 2}{2} + 1\right) = k - 1$ and $\frac{3k^2 - k + 2}{2} - \left(\frac{3k^2 - 3k + 2}{2} + 1\right)  = k - 1$, we reindex:
    \begin{align*}
        \frac{d}{dx}\ln(g(x))\bigg|_{x=k} &= -2\left(\sum_{i=\frac{3k^2-k+2}{2}+1}^{\frac{3k^2+k+2}{2}}\frac{1}{i}\right)+ \left(\frac{6k-3}{2}\right)\left( \sum_{i=\frac{3k^2-3k+2}{2} + 1}^{\frac{3k^2-k+2}{2}}\frac{1}{i}-\sum_{i=\frac{3k^2-k+2}{2}+1}^{\frac{3k^2+k+2}{2}}\frac{1}{i}\right) \\
        &= -2\left(\sum_{i=1}^{k}\frac{1}{\frac{3k^2-k+2}{2}+i}\right)+ \left(\frac{6k-3}{2}\right)\left( \sum_{i=1}^{k}\frac{1}{\frac{3k^2-3k+2}{2}+i}-\sum_{i=1}^{k}\frac{1}{\frac{3k^2-k+2}{2}+i}\right) \\
        &=-2\left(\sum_{i=1}^{k}\frac{2}{3k^2-k+2+2i}\right)+ \left(\frac{6k-3}{2}\right)\left(\sum_{i=1}^{k}\frac{2}{3k^2-3k+2+2i}-\sum_{i=1}^{k}\frac{2}{3k^2-k+2+2i}\right)\\
        &=\sum_{i=1}^{k}\frac{-4}{3k^2-k+2+2i}+ \left(6k-3\right)\left(\sum_{i=1}^{k}\frac{3k^2-k+2+2i-(3k^2-3k+2+2i)}{(3k^2-3k+2+2i)(3k^2-k+2+2i)}\right)\\
        &=\sum_{i=1}^{k}\frac{-4(3k^2 - 3k + 2 + 2i)}{(3k^2 - 3k + 2 + 2i)(3k^2-k+2+2i)}+ \sum_{i=1}^{k}\frac{\left(6k-3\right)\cdot 2k}{(3k^2-3k+2+2i)(3k^2-k+2+2i)}\\
        &=\sum_{i=1}^{k}\frac{6k-8-8i}{(3k^2-3k+2+2i)(3k^2-k+2+2i)}
    \end{align*}
    and we have achieved our goal of rewriting $\frac{d}{dx}\ln(g(x))\big|_{x=k}$ as a single summation.
    The denominator of this summation is positive for all $k \geq 1$, so the numerator determines the sign of the sum.
    Specifically, we restrict to when $k \geq 8$, so that $6k - 8 \geq 4k + 8 = 8 \cdot \left(\frac{k}{2} + 1\right)$, and thus the numerator of our summation is positive when $i \leq \frac{k+1}{2} < \frac{k}{2} + 1$.
    Let $j \in [n]$, $1 \leq j \leq \frac{k+1}{2}$, then, and consider the terms of the summation with $i=j$ and $i = k-j+1$.
    
    As $j \leq \frac{k+1}{2}$, the $i=j$ term of the summation is positive; if the $i=k-j+1$ term is nonnegative, then the sum of these two terms is positive. Otherwise,  note that for $j\leq \frac{k+1}{2}$, we have $k-1+j\geq j$, such that if the $i=k-j+1$ term is negative, then
    $$0>\frac{6k-8-8(k-j+1)}{(3k^2-3k+2+2(k-j+1))(3k^2-k+2+2(k-j+1))} \geq \frac{6k-8-8(k-j+1)}{(3k^2-3k+2+2j)(3k^2-k+2+2j)}$$
    and thus the sum of the $i=j$ and $i=k-j+1$ terms is
    \begin{gather*}
        \frac{6k-8-8j}{(3k^2-3k+2+2j)(3k^2-k+2+2j)} + \frac{6k-8-8(k-j+1)}{(3k^2-3k+2+2(k-j+1))(3k^2-k+2+2(k-j+1))} \\\geq \frac{6k-8-8j}{(3k^2-3k+2+2j)(3k^2-k+2+2j)} + \frac{6k-8-8(k-j+1)}{(3k^2-3k+2+2j)(3k^2-k+2+2j)}\\= \frac{4k-24}{(3k^2-3k+2+2j)(3k^2-k+2+2j)},
    \end{gather*}
    which is always positive for $k \geq 8$. 
    By pairing the $i$th and $(k-i+1)$th terms, $i < \frac{k+1}{2}$, of the summation we obtain a sum of only positive terms. 
    If $k$ is even, all terms of the summation get paired, and thus $\frac{d}{dx}\ln(g(x))\big|_{x=k} > 0$; if $k$ is odd, then the $i = \frac{k+1}{2}$th term does not get paired.
    Specifically, we have the sum
    \begin{equation*}
        \sum_{i=1}^{\ceil{\frac{k+1}{2}}-1}\frac{6k-8-8i}{(3k^2-3k+2+2i)(3k^2-k+2+2i)} + \sum_{i = \floor{\frac{k+1}{2}}+1}^{k}\frac{6k-8-8i}{(3k^2-3k+2+2i)(3k^2-k+2+2i)} > 0.
    \end{equation*}
    As shown before, though, when $k \geq 8$, we have the numerator of the $\frac{k+1}{2}$th term nonnegative, making the whole term nonnegative, so that $\frac{d}{dx}\ln(g(x))\big|_{x=k} > 0$ for all $k \geq 8$; that is, $\ln(g(x))\big|_{x=k}$ is increasing for $k \geq 8$.
   
    Thus, when $g(x)$ is restricted to $x=k\in\NN$ it is increasing for $k \geq 8$, as a function is increasing if and only if its natural logarithm is also.
    One may observe that, when $x=k\in \NN$, $g(k) = \frac{{n-k\choose k}}{{n\choose k}}> \frac{1}{2}$ for $1 \leq k \leq 8$ (\cref{fig:original bound g}). Therefore, we have that $g(k)>\frac{1}{2}$ for all $k \geq 1$.
    Recalling that when $n = \frac{3k^2 + k + 2}{2}$ and we have $\binom{n-k}{k}>\frac{1}{2}\binom{n}{k}$ for all $k \geq 1$ by \cref{lem:characteristics of f(n)}, we conclude that $\sn(\KG(n,k))=\gon(\KG(n,k))=\binom{n-1}{k}$ for all $n \geq \frac{3k^2 + k + 2}{2}$ by \cref{lem: n > 2k then sn = gon bound}.
\end{proof}

\begin{figure}[h]\centering
    \hfill
    \begin{minipage}{0.5\textwidth}
    \includegraphics[]{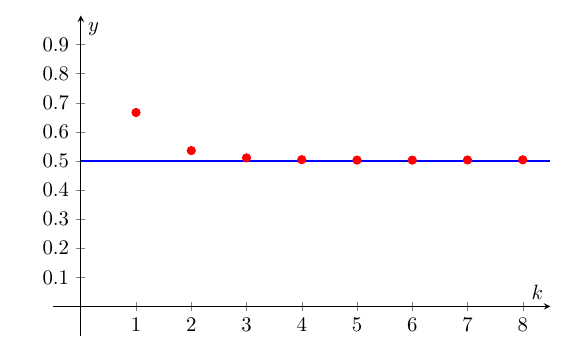}
    \end{minipage}
    \hfill\renewcommand{\arraystretch}{1.25}
    \begin{tabular}{c|c}
    $k$ & $g(k)$ \\
    \hline
    1 & 0.6667 \\
    2 & 0.5357 \\
    3 & 0.5107 \\
    4 & 0.5046 \\
    5 & 0.5031 \\
    6 & 0.5030 \\
    7 & 0.5034 \\
    8 & 0.5039
    \end{tabular}
    \hfill\hfill
    
    \caption{Small values of $g(k) = \frac{\binom{n-k}{k}}{\binom{n}{k}}$ for $n = \frac{3k^2 + k + 2}{2}$} rounded to 4 decimal places.
    \label{fig:original bound g}
\end{figure}

As a result of this, we now revisit the polynomial bound stated by Liu, Cao, and Lu in \cite{liu2021treewidthgeneralizedknesergraphs}, given by Theorem \ref{lem:tw poly bound}. Not only is the polynomial bound for $n$ in terms of $k$ in Theorem \ref{thm: 3k^2+k lower bound} an improved bound, but the theorem states that Kneser graphs will have gonality exactly equal to $\binom{n-1}{k}$ for all $n\geq \frac{3k^2+k+2}{2}$, as opposed to treewidth being equal to $\binom{n-1}{k}-1$ for $n\geq 4k^2-3k+2$. We formally prove our polynomial bound is an improvement in the following proposition for all values of $k\geq 1$.

\begin{proposition}
    For $k\geq 1$, $$\frac{3k^2+k+2}{2}\leq 4k^2-3k+2.$$
\end{proposition}

\begin{proof}
    Observe that for $k\geq 1$, we have that $k-1\geq 0$ and $5k-2\geq 3>0$. Then we have that $$0\leq(k-1)(5k-2).$$ Moreover, this gives us that \begin{align*}
        0&\leq(k-1)(5k-2)\\
        0&\leq 5k^2-7k+2\\
        (3k^2+k+2)+0&\leq (3k^2+k+2)+5k^2-7k+2\\
        3k^2+k+2&\leq8k^2-6k+4\\
        \frac{3k^2+k+2}{2}&\leq 4k^2-3k+2.
    \end{align*}
    Thus proving our inequality as desired.
\end{proof}

In Figure 7 we visualize our polynomial bound compared to the polynomial bound in \cref{lem:tw poly bound}.

\begin{figure}[h]\centering
\includegraphics[]{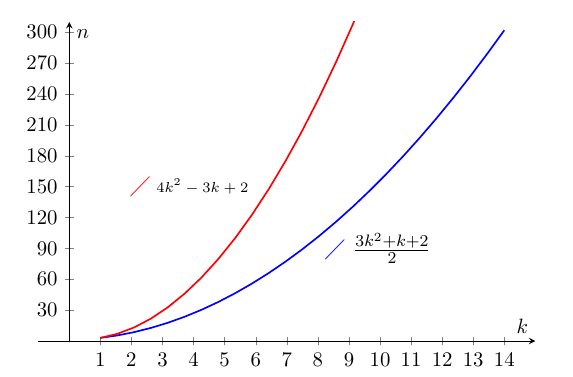}
    \caption{Comparison of the polynomial bound given in \cref{lem:tw poly bound} (red) and the polynomial bound set out in this paper (blue).}
    \label{fig:polynomial comparison}
\end{figure}

\subsection{Generalized Kneser Graphs}
\label{sec: gen kneser}
\begin{definition}
Let $n\geq k \geq t \geq 1$. The \emph{generalized Kneser graph} $\KG(n,k,t)$ is the graph whose vertices correspond to the elements of $\binom{\left[n\right]}{k}$ and there is an edge between two vertices $v_A$ and $v_B$ if the corresponding subsets they represent satisfy $|A\cap B|<t$.
\end{definition}

As such, $KG(n,k,1) = KG(n,k)$. Let us now write the analog properties for this graph. First, we note that the vertex set remains the same, and so the number of vertices is still $\binom{n}{k}$. However, there are now more edges we have added. To understand this, let’s compute the degree of a vertex.

The neighborhood of a vertex $A$ is given by $\bigcup_{i=0}^{t-1} \{B\in \binom{\left[n\right]}{k} \mid |A\cap B| = i\}$.
As such, we may make a counting argument to obtain the following quantity $$\Delta(KG(n,k,t)) = \sum_{i=0}^{t-1} \binom{k}{i} \binom{n-k}{k-i}.$$

When stated in its full generality, the Erd\H{o}s-Ko-Rado theorem implies that
$\alpha(KG(n,k,t)) = \binom{n-t}{k-t}$ \cite{EKR}.
Altogether, we have the pieces necessary to state a sufficient bound for gonality.

\begin{corollary}
	If $\sum_{i=0}^{t-1} \binom{k}{i} \binom{n-k}{k-i} > \frac{1}{2} \binom{n}{k}$ and $n>2k$, then $$\sn(KG(n,k,t)) = \gon(KG(n,k,t)) = \binom{n}{k}-\binom{n-t}{k-t}.$$
\end{corollary}


We then obtain the following result. 
\begin{proof}[Proof of \cref{thm: gen kneser}]
    Let $n, k$ and $t$ be as stated above. Recall that in the proof of \Cref{{thm: 3k^2+k lower bound}}, we have that if $n\geq \frac{3k^2+k+2}{2}$, then 
    $$\binom{n-k}{k}>\frac{1}{2}\binom{n}{k}.$$
    As such, we obtain
    \begin{align*}
        \sum_{i=0}^{t-1} \binom{k}{i} \binom{n-k}{k-i} = \sum_{i=1}^{t-1} \binom{k}{i} \binom{n-k}{k-i}+\binom{n-k}{k}  \geq \binom{n-k}{k}>\frac{1}{2}\binom{n}{k}.
    \end{align*}
    Then by \Cref{{lem:sn = gon for simple graphs}}, we have
    \begin{equation*}\sn(KG(n,k,t)) =\gon(KG(n,k,t)) = \binom{n}{k}-\binom{n-t}{k-t}.\qedhere\end{equation*}
\end{proof}

\begin{remark}
    In the case of generalized Kneser graphs $\KG(n,k,t)$, we discovered that even a slight alteration to the polynomial bound given in \Cref{thm: gen kneser} is not trivial to prove. One such candidate polynomial is $n=\frac{3k^2+k+2}{t}$. We note that when $t=1$, $\frac{3k^2+k+2}{2}<3k^2+k+2$. But for $2\leq t\leq k$, with arbitrary fixed $k$ such that $k,t\in \NN$, we would have to prove the inequality \begin{align*}
        \sum_{i=0}^{t-1} \frac{\binom{k}{i} \binom{n-k}{k-i}}{\binom{n}{k}}=\sum_{i=0}^{t-1}\frac{\left(\left(\frac{3k^2-k(t-1)+2}{t}\right)!\right)^2 (k!)^2}{\left(\frac{3k^2-k(2t-1)+it+2}{t}\right)!\left(\frac{3k^2+k+2}{t}\right)!((k-i)!)^2 \cdot i!}>\frac{1}{2}.
    \end{align*}
    That is, the parameter of $t$ affects not only the number of terms but also the input of the factorials in the expression. Additionally, the terms are not monotone increasing, following a concave pattern where smallest terms is obtained when $i$ is maximized at $i = k-1$.

    Since these are rational functions, we observe that the factorials do not always have integer inputs (consider $k=t=3$), and so we use the smooth extension via the Gamma function to evaluate the factorial functions of $k$ and $t$. 
    An important property of the Gamma function is that if $x>0$, then $\Gamma(x)>0$.  
    One can show that each factor above is positive for all $2\leq t\leq k$ by showing that the polynomials in $k$ and $t$ have no critical points defined in their domain, and by showing that at the boundaries of their domains, each polynomial is positive (at $t=2, t=k,$ and as $k$ tends to infinity). Thus, if one shows that a finite number of terms of the summation is greater that $\frac{1}{2}$, then we have that the entire summation is greater than $\frac{1}{2}$. However, as $k$ increases, in evaluating each term of the summation as $2\leq t\leq k$, we found that it takes an increasing number of terms in the summation for this to be true, which is not trivial to show.
\end{remark}

Recall that the treewidth of these graphs has already been studied. For $t>1$, in \cite{ajc-v86-p477}, Metsch produced a significant improvement of the previously known bound by Liu, Cao and Lu. Below we state Metsch's theorem regarding treewidth which shows that we have a similar situation as with the $t=1$ case:
\begin{theorem}[\cite{ajc-v86-p477}*{Theorem 1.4}]
\label{lem:gen tw poly bound}

For integers $n,k,t$ with $k > t > 1$ and
$$n\geq \begin{cases}
    t +6k(k +1-t)(k -t)
 & \text{ if } 2 \leq t \leq 16\\
t + \frac{t-1}{\ln(t)}
k(k + 1 - t)(k - t) &\text{ if } t \geq 17
\end{cases}
$$
we have

$$\tw(\KG(n,k,t)) =\binom{n}{k}-\binom{n-t}{k-t}-1.$$

    If $n\geq 2(k-t)(t+1)\binom{k}{t}+k+t+1$. Let $k>t\geq 2$
    \begin{align*}
        \text{tw}(\KG(n,k,t)) =\binom{n}{k}-\binom{n-t}{k-t}-1.
    \end{align*}  
\end{theorem}

As a consequence of both Theorem \ref{thm: gen kneser} and Theorem \ref{lem:gen tw poly bound} we obtain:

\begin{proof}[Proof of \cref{coll: inf subs}]
Let $n>2k$ and $k>t>0$.
    \begin{enumerate}
        \item Note that for $2\leq t\leq 16$ and each $k>t$, if $n\geq t +6k(k +1-t)(k -t)$ then we have $\gon(KG(n,k,t)) = \sn(KG(n,k,t)) =\tw(KG(n,k,t))+1.$ Similarly, for $t\geq 17$ and each $k>t$, if $n\geq t + \frac{t-1}{\ln(t)}
k(k + 1 - t)(k - t)$ then $\gon(KG(n,k,t)) = \sn(KG(n,k,t)) =\tw(KG(n,k,t))+1.$
        \item Note that for each value of $t$ greater than $1$, we obtain that the treewidth polynomial bound is a polynomial of degree $3$. However, our bound for scramble number is a degree $2$ polynomial. Altogether, this implies that for all $t$, there exists an infinite collection of pairs $(k,n)$ for which the treewidth is not known, however the scramble number, and subsequently the gonality are known to be $\binom{n}{k}-\binom{n-t}{k-t}$.
    \end{enumerate}    
\end{proof}

We remark that that for the extremal value of $t = k$, we have that the graph is the complete graph, hence the polynomial bound is $n>2k$, a linear polynomial. Since the degree of a vertex increases with the value of $t$, we know that each time we increase the value of $t$, we obtain a smaller polynomial bound. That is, as we increase the minimal degree of a vertex, the faster we can conclude that the scramble number determines the gonality. Hence, this leaves the open question of, can a description be given of this family of decreasing polynomials?

\section{Conjectured Improvement to our bound}
\label{sec: conj improvement}
\subsection{Conjectured Bound}
A natural question that follows from Theorem \ref{thm: 3k^2+k lower bound} is ``Can we improve the polynomial bound on $n$ in terms of $k$ for which $\gon(\KG(n,k))=\binom{n-1}{k}$?'' To answer this, we considered various scrambles of the Kneser graphs $\KG(5,2), \KG(6,2), \KG(7,2),$ and $\KG(7,3)$. One scramble of particular interest is what we call the \emph{uniform edge scramble} \cite{cenek2023uniformscramblesgraphs}, which is a scramble consisting of an egg around every individual edge in the graph. We observe that the hitting size for this scramble is always equal to $|V|-\alpha(\graph) = \binom{n-1}{k}$, which is consistent with the upper bound for gonality given by Corollary \ref{lem:gon < n - alpha for simple graphs}. If we could show that the egg-cut for the uniform edge scramble is no smaller than the hitting set, then we would have that $\sn(\KG(n,k))=\gon(\KG(n,k))=\binom{n-1}{k}$. 

A natural candidate for an egg-cut in the uniform edge scramble is given by disconnecting the $2\cdot(\binom{n-k}{k}-1)$ edges in the neighborhood of an edge. However, this is not necessarily the optimal way to produce an egg-cut. This property is known in the literature as $\lambda_2$-optimality\cite{ESFAHANIAN1988195}, which was studied for Kneser graphs in \cite{BALBUENA2019258}. There they concluded $\lambda_2$-optimality for Kneser graphs with $k=2$. However, it is unknown for $k>2$, whether Kneser graphs are $\lambda_2$-optimal or not. We proceed in this section to consider the improvement we can do to our bound if Kneser graphs satisfy this property. That is, we consider when $$2\cdot\left(\binom{n-k}{k}-1\right)\geq \binom{n-1}{k},$$ which simplifies to $$\binom{n-k}{k}>\frac{1}{2}\binom{n-1}{k}.$$ As we did in Section \ref{sec: main work}, we considered the minimum $n$ satisfying this inequality for a fixed $k$ and found empirically that this held for $n\geq \frac{3k^2-3k+10}{2}$. 

As we did in Section \ref{sec: calc}, we examine the behavior of the expression $\frac{\binom{n-k}{k}}{\binom{n-1}{k}}$. 
In a similar manner to our main theorem, we require the following Corollary to Lemma \ref{lem:characteristics of f(n)} to show $\frac{\binom{n-k}{k}}{\binom{n-1}{k}}$ is increasing in $n$ for a fixed $k$.
\begin{corollary}\label{cor:characteristics modified f(n)}
    For any fixed integer $k\geq 2$, let $\displaystyle f(x) = \frac{\binom{x-k}{k}}{\binom{x-1}{k}} $. Then $\{f(n)\}_{n=1}^{\infty}$
    is strictly increasing whenever $n>2k$.
\end{corollary}
\begin{proof}
    Let $g(x) = \frac{\binom{x-k}{k}}{\binom{x}{k}}$, note that $f(x) = \displaystyle \frac{\binom{x-k}{k}}{\binom{x-1}{k}} = \frac{x}{x-k}\cdot g(x)$.
    Then, by \cref{lem:characteristics of f(n)}, $$\frac{df}{dx} = \frac{(x-k) - x}{(x-k)^2}g(x) + \frac{x}{x-k}\cdot g'(x) = -\frac{k}{(x-k)^2}g(x) + \frac{x}{x-k}g'(x),$$
    where $g'(x) > 0$ for all $x > 2k$. Furthermore, 

    $$g'(x) = k\left(\prod_{i=0}^{k-1}\frac{x-k-i}{x-i}\right)\left(\sum_{i=0}^{k-1}\frac{1}{(x-i)(x-k-i)}\right),$$

    $$\frac{x}{x-k}g'(n) = k\left(\prod_{i=1}^{k-1}\frac{x-k-i}{x-i}\right)\left(\sum_{i=0}^{k-1}\frac{1}{(x-i)(x-k-i)}\right),$$ and

    $$g(x) = \prod_{i=0}^{k-1}\frac{x-k-i}{x-i}.$$ Thus,

    \begin{align*}
        f'(x) &= \prod_{i=0}^{k-1}\frac{x-k-i}{x-i}\left(\frac{-k}{(x-k)^2} + \frac{xk}{x-k}\left(\sum_{i=0}^{k-1}\frac{1}{(x-i)(x-k-i)}\right)\right) \\
        &= \frac{k}{x-k}\prod_{i=0}^{k-1}\frac{x-k-i}{x-i}\left(\frac{-1}{x-k} + x \sum_{i=0}^{k-1}\frac{1}{(x-i)(x-k-i)}\right) \\
        &= \frac{k}{x-k}\prod_{i=0}^{k-1}\frac{x-k-i}{x-i}\left(\frac{-1}{x-k} + \frac{x}{x(x-k)} + x \sum_{i=1}^{k-1}\frac{1}{(x-i)(x-k-i)}\right) \\
        &= \frac{k}{x-k}\prod_{i=0}^{k-1}\frac{x-k-i}{x-i}\left(x \sum_{i=1}^{k-1}\frac{1}{(x-i)(x-k-i)}\right) \\
        &= \frac{xk}{x-k}\prod_{i=0}^{k-1}\frac{x-k-i}{x-i}\sum_{i=1}^{k-1}\frac{1}{(x-i)(x-k-i)}.
    \end{align*}
    Note that since $x>2k$, then $x-i>0$ and $x-k-i>0$ for $0\leq i \leq k-1$. Therefore, $f'(x)>0$ and we conclude that $f(x)$ is increasing when $x>2k$.
\end{proof}

Given that $n=g(x)=\frac{3x^2-3x+10}{2}$ satisfies the necessary conditions, we are able to take the smooth extension of $g$ as defined in Section \ref{sec: calc}, and when $x$ is restricted to $k\in\NN$ we apply \Cref{cor:macaroni annihilation} in taking the derivative of $\ln\left(\frac{\binom{n-k}{k}}{\binom{n-1}{k}}\right)\big|_{x=k}$ when $n=\frac{3x^2-3x+10}{2}$.

We now prove that the inequality $\binom{n-k}{k}>\frac{1}{2}\binom{n-1}{k}$ holds for $n\geq \frac{3k^2-3k+10}{2}$, which improves upon the polynomial bound for $n$ in terms of $k$ given by Theorem \ref{thm: 3k^2+k lower bound}.

\begin{proposition}
    \label{prop: conj poly bound} For $k\geq 1$ and $n\geq \frac{3k^2-3k+10}{2}$, $$\binom{n-k}{k}>\frac{1}{2}\binom{n-1}{k}.$$
\end{proposition}

We approach the proof of this \lcnamecref{prop: conj poly bound} in a similar manner to that of \cref{thm: 3k^2+k lower bound}: For a fixed $k$, \Cref{cor:characteristics modified f(n)} yields that if $\binom{N-k}{k} > \frac{1}{2}\binom{N-1}{k}$ holds for $N$, then $\binom{n-k}{k} > \frac{1}{2}\binom{n-1}{k}$ holds for all $n \geq N$.
Here we show that when $n = \frac{3x^2 - 3x + 10}{2}$, the inequality $\frac{\binom{n-x}{k}}{\binom{n-1}{x}}>\frac{1}{2}$ holds for when $x=k\in \NN$ for $k\geq 1$.

\begin{proof}[Proof of \cref{prop: conj poly bound}]
    Fix $n = \frac{3x^2 - 3x + 10}{2}$ and let $g(x)= \frac{\binom{n-x}{x}}{\binom{n-1}{x}}$. 
    Then $$g(x) = \frac{{\frac{3x^2 - 3x + 10}{2} - x\choose x}}{{\frac{3x^2 - 3x + 10}{2} -1\choose x}} 
    = \frac{{\frac{3x^2 - 5x + 10}{2} \choose x}}{{\frac{3x^2 - 3x + 8}{2}\choose x}} 
    = \frac{\left(\frac{3x^2 - 5x + 10}{2}\right)!}{\left(\frac{3x^2 - 7x + 10}{2}\right)!x!}\cdot\frac{\left(\frac{3x^2 - 5x + 8}{2}\right)!x!}{\left(\frac{3x^2 - 3x + 8}{2}\right)!} 
    = \frac{\left(\frac{3x^2 - 5x + 10}{2}\right)!}{\left(\frac{3x^2 - 7x + 10}{2}\right)!}\cdot\frac{\left(\frac{3x^2 - 5x + 8}{2}\right)!}{\left(\frac{3x^2 - 3x + 8}{2}\right)!}.$$

    Since $g(x)$ satisfies the necessary conditions, we restrict $x$ to $k\in \NN$ and apply \Cref{cor:macaroni annihilation}, obtaining
    \begin{align*}
        \frac{d}{dx}\ln(g(x))\bigg|_{x=k}&=\left(\frac{6k-5}{2}\right)\left(\sum_{i=1}^{\frac{3k^2-5k+10}{2}}\frac{1}{i}\right)+\left(\frac{6k-5}{2}\right)\left(\sum_{i=1}^{\frac{3k^2-5k+8}{2}}\frac{1}{i}\right)\\
        &\qquad-\left(\frac{6k-7}{2}\right)\left(\sum_{i=1}^{\frac{3k^2-7k+10}{2}}\frac{1}{i}\right)-\left(\frac{6k-3}{2}\right)\left(\sum_{i=1}^{\frac{3k^2-3k+8}{2}}\frac{1}{i}\right).
    \end{align*}
    Extracting a factor of $\frac{6k - 7}{2}$ from each term and grouping these yields
    \begin{align*}
        \frac{d}{dx}\ln(g(x))\bigg|_{x=k} &=\left(\frac{6k-7}{2}\right) \left(\sum_{i=1}^{\frac{3k^2-5k+10}{2}}\frac{1}{i}+\sum_{i=1}^{\frac{3k^2-5k+8}{2}}\frac{1}{i}-\sum_{i=1}^{\frac{3k^2-7k+10}{2}}\frac{1}{i}-\sum_{i=1}^{\frac{3k^2-3k+8}{2}}\frac{1}{i}\right)\\
        &\qquad+\left(\sum_{i=1}^{\frac{3k^2-5k+10}{2}}\frac{1}{i}+\sum_{i=1}^{\frac{3k^2-5k+8}{2}}\frac{1}{i}-2\cdot\sum_{i=1}^{\frac{3k^2-3k+8}{2}}\frac{1}{i}\right)\\
        &=\left(\frac{6k-7}{2}\right)\left(\sum_{i=\frac{3k^2-7k+10}{2}+1}^{\frac{3k^2-5k+10}{2}}\frac{1}{i}-\sum_{i=\frac{3k^2-5k+8}{2}+1}^{\frac{3k^2-3k+8}{2}}\frac{1}{i}\right)-\left(\sum_{i=\frac{3k^2-5k+8}{2}+1}^{\frac{3k^2-3k+8}{2}}\frac{1}{i}+\sum_{i=\frac{3k^2-5k+10}{2}+1}^{\frac{3k^2-3k+8}{2}}\frac{1}{i}\right).
    \end{align*}
    We will coalesce these summations into two summations, and then demonstrate that their difference is positive.
    Considering each summand of the derivative individually, we first note that $\frac{3k^2 - 5k + 10}{2} - \left(\frac{3k^2 - 7k + 10}{2} + 1\right) = k - 1$ and $\frac{3k^2 - 3k + 8}{2} - \left(\frac{3k^2 - 5k + 8}{2} + 1\right) = k - 1$, so that
    \begin{align*}
        \left(\frac{6k-7}{2}\right)\left(\sum_{i=\frac{3k^2-7k+10}{2}+1}^{\frac{3k^2-5k+10}{2}}\frac{1}{i}-\sum_{i=\frac{3k^2-5k+8}{2}+1}^{\frac{3k^2-3k+8}{2}}\frac{1}{i}\right) &= \left(\frac{6k-7}{2}\right)\sum_{i=1}^{k}\left(\frac{1}{\frac{3k^2-7k+10}{2}+i}-\frac{1}{\frac{3k^2-5k+8}{2}+i}\right)\\
        &\hspace{-4em}= \left(\frac{6k-7}{2}\right)\sum_{i=1}^{k}\left(\frac{2}{3k^2-7k+10+2i}-\frac{2}{3k^2-5k+8+2i}\right) \\
        &= \sum_{i=1}^{k}\left(\frac{6k-7}{3k^2-7k+10+2i}-\frac{6k-7}{3k^2-5k+8+2i}\right).
    \end{align*}
    We reindex the second summand similarly, using the fact that $\frac{3k^2 - 3k + 8}{2} - \left(\frac{3k^2 - 5k + 10}{2} + 1\right) = k - 2$ to extract an $i = \frac{3k^2 - 5k + 8}{2} + 1 = \frac{3k^2 - 5k + 10}{2}$ term:
    \begin{align*}
        \sum_{i=\frac{3k^2-5k+8}{2}+1}^{\frac{3k^2-3k+8}{2}}\frac{1}{i}+\sum_{i=\frac{3k^2-5k+10}{2}+1}^{\frac{3k^2-3k+8}{2}}\frac{1}{i} &= \sum_{i=\frac{3k^2-5k+8}{2}+1}^{\frac{3k^2-3k+8}{2}}\frac{1}{i}+\left(\sum_{i=\frac{3k^2-5k+8}{2}+1}^{\frac{3k^2-3k+8}{2}}\frac{1}{i} - \frac{1}{\frac{3k^2-5k+8}{2}+1}\right) \\
        &= 2\cdot\sum_{i=\frac{3k^2-5k+8}{2}+1}^{\frac{3k^2-3k+8}{2}}\frac{1}{i}-\frac{2}{3k^2-5k+10} \\
        &= 2\sum_{i=1}^{k}\frac{1}{\frac{3k^2-5k+8}{2}+i}-\frac{2}{3k^2-5k+10} \\
        &= 2\sum_{i=1}^{k}\frac{2}{3k^2-5k+8+2i}-\frac{2}{3k^2-5k+10}
    \end{align*}

    Therefore, we have
    \begin{align*}
        \frac{d}{dx}&\ln(g(x))\bigg|_{x=k} = \sum_{i=1}^{k}\left(\frac{6k-7}{3k^2-7k+10+2i}-\frac{6k-7}{3k^2-5k+8+2i}\right) \\
        &\hspace{10em} - \left(\left(\sum_{i=1}^{k}\frac{4}{3k^2-5k+8+2i}\right)-\frac{2}{3k^2-5k+10}\right)\\
        &= \sum_{i=1}^{k}\left(\frac{6k-7}{3k^2-7k+10+2i}-\frac{6k-3}{3k^2-5k+8+2i}\right) + \frac{2}{3k^2-5k+10} \\
        &\geq \sum_{i=1}^{k}\left(\frac{6k-7}{3k^2-7k+10+2i}-\frac{6k-3}{3k^2-5k+8+2i}\right) + \frac{1}{k}\sum_{i=1}^k\frac{2}{3k^2-5k+8 + 2i} \\
        &= \sum_{i=1}^{k}\left(\frac{6k-7}{3k^2-7k+10+2i}-\frac{6k-3 - 2/k}{3k^2-5k+8+2i}\right)\\
        &= (6k-7)\sum_{i=1}^{k}\left(\frac{1}{3k^2-7k+10+2i}-\frac{1}{3k^2-5k+8+2i}\right) - \sum_{i=1}^{k}\frac{4-2/k}{3k^2-7k+10+2i} \\
        &= \frac{1}{k}(6k^2-7k)\sum_{i=1}^{k}\left(\frac{2k - 2}{(3k^2 - 7k + 10 + 2i)(3k^2 - 5k + 8 + 2i)}\right) - \frac{1}{k}\sum_{i=1}^{k}\frac{4k-2}{3k^2-7k+10+2i} \\
        &= \frac{1}{k}\sum_{i=1}^k\frac{2(k-1)(6k^2-7k)}{(3k^2 - 7k + 10 + 2i)(3k^2 - 5k + 8 + 2i)} - \frac{1}{k}\sum_{i=1}^{k}\frac{2(2k-1)}{3k^2-7k+10+2i} \\
        &= \frac{2}{k}\sum_{i=1}^k\left(\frac{(k-1)(6k^2-7k)}{(3k^2 - 7k + 10 + 2i)(3k^2 - 5k + 8 + 2i)} - \frac{(2k-1)(3k^2 - 7k + 10 + 2i)}{(3k^2-7k+10+2i)(3k^2 - 5k + 8 + 2i)}\right) \\
        &= \frac{2}{k}\sum_{i=1}^k\frac{4k^2 - 20k + 10 + (-4k + 2)i}{(3k^2-7k+10+2i)(3k^2 - 5k + 8 + 2i)} \\
    \end{align*}

    As before, the denominator of this summation is positive for all $k \geq 1$, so the numerator determines the sign of the sum.
    This time, we restrict to when $k \geq 11$, so that $$4k^2 - 20k + 10 + (-4k+2)\left(\frac{k}{2}+1\right)> 0.$$  
    Then, because $4k^2 - 20k + 10 + (-4k + 2)i$ is decreasing in $i$ when $k \geq 1$, we have that the numerator of our summation is positive when $i \leq \frac{k+1}{2} < \frac{k}{2} + 1$.
    
    Now consider the terms of the summation with $i = j$ and $i = k-j+1$, where $j \in [n]$ and $1 \leq j \leq \frac{k+1}{2}$.
    Because $j \leq \frac{k+1}{2}$, the $i=j$ term of the summation is positive; if the $i=k-j+1$ term is nonnegative, then the sum of these two terms is positive. 
    Otherwise,  note that for $j\leq \frac{k+1}{2}$, we have $k-j+1\geq j$, such that if the $i=k-j+1$ term is negative, then 
    \begin{align*}
        0 &>\frac{4k^2 - 20k + 10 + (-4k + 2)(k-j+1)}{(3k^2-7k+10+2(k-j+1))(3k^2 - 5k + 8 + 2(k-j+1))} \\
        &\geq \frac{4k^2 - 20k + 10 + (-4k + 2)(k-j+1)}{(3k^2-7k+10+2j)(3k^2 - 5k + 8 + 2j)}
    \end{align*}
    and thus the sum of the $i = j$ and $i = k-j+1$ terms is
    \begin{gather*}
        \frac{4k^2 - 20k + 10 + (-4k + 2)j}{(3k^2-7k+10+2j)(3k^2 - 5k + 8 + 2j)} + \frac{4k^2 - 20k + 10 + (-4k + 2)(k-j+1)}{(3k^2-7k+10+2(k-j+1))(3k^2 - 5k + 8 + 2(k-j+1))}
        \\\geq \frac{4k^2 - 20k + 10 + (-4k + 2)j}{(3k^2-7k+10+2j)(3k^2 - 5k + 8 + 2j)} + \frac{4k^2 - 20k + 10 + (-4k + 2)(k-j+1)}{(3k^2-7k+10+2j)(3k^2 - 5k + 8 + 2j)}
        \\= \frac{4k^2 - 42k + 22}{(3k^2-3k+2+2i)(3k^2-k+2+2i)},
    \end{gather*}
    
    which is always positive for $k \geq 11$. 
    By pairing the $i$th and $(k-i+1)$th terms, $i < \frac{k+1}{2}$, of the summation we obtain a sum of only positive terms. 
    If $k$ is even, all terms of the summation get paired, and thus the summation is positive; if $k$ is odd, then the $i = \frac{k+1}{2}$th term does not get paired, but all other pairings are strictly positive.

    As demonstrated, when $k \geq 10$, we have the $\frac{k+1}{2}$th term nonnegative, so the full summation is positive.
    In all cases, we have $$\frac{d}{dx}\ln(g(x))\bigg|_{x=k} \geq \frac{2}{k}\sum_{i=1}^k\frac{4k^2 - 20k + 10 + (-4k + 2)i}{(3k^2-7k+10+2i)(3k^2 - 5k + 8 + 2i)} > 0,$$ so when $x$ is restricted to $k\in \NN$, $\ln(g(k))$ is increasing for $k \geq 11$.
    It follows that $g(k)$ is increasing for $k \geq 11$.
    
    One may observe that, when $x=k\in\NN$ that $g(x) > \frac{1}{2}$ for $1 \leq k \leq 11$ (\cref{fig:conjecture bound g}). Therefore, we have that $g(k)>\frac{1}{2}$ for all $k \geq 1$.
    We conclude that when $x=k\in\NN,$ if $n = \frac{3k^2 - 3k + 10}{2}$, then $g(k)=\frac{{n-k\choose k}}{{n\choose k}}>\frac{1}{2}$, giving us $\binom{n-k}{k}>\frac{1}{2}\binom{n-1}{k}$ for all $k \geq 1$ and $n \geq \frac{3k^2 - 3k + 10}{2}$ by \cref{cor:characteristics modified f(n)}.
\end{proof}

\begin{figure}[h]\centering
    \hfill
    \begin{minipage}{0.7\textwidth}
    \includegraphics[]{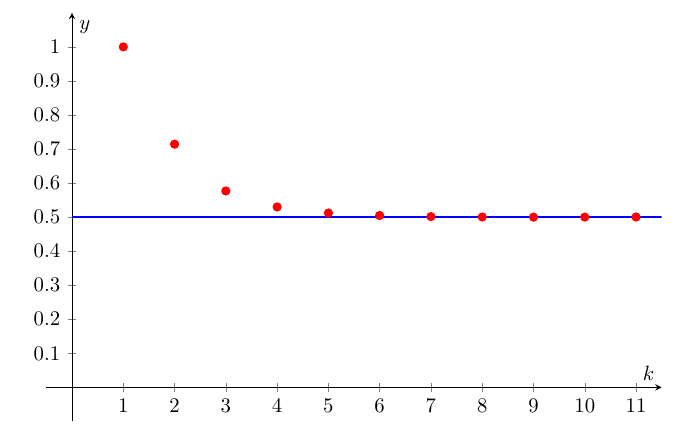}
    \end{minipage}
    \hfill\renewcommand{\arraystretch}{1.25}
    \begin{tabular}{c|c}
    $k$ & $g(k)$ \\
    \hline
    1 & 1.00000 \\
    2 & 0.71429 \\
    3 & 0.57692 \\
    4 & 0.52987 \\
    5 & 0.51214 \\
    6 & 0.50480 \\
    7 & 0.50166 \\
    8 & 0.50040 \\
    9 & 0.50003 \\
    10 & 0.50011 \\
    11 & 0.50041
    \end{tabular}
    \hfill\hfill
    
    \caption{Small values of $g(k) = \frac{\binom{n-k}{k}}{\binom{n-1}{k}}$ for $n = \frac{3k^2 - 3k + 10}{2}$, rounded to 5 decimal places.}
    \label{fig:conjecture bound g}
\end{figure}

\subsection{$\lambda_2$-optimality}
We know that the uniform edge scramble has a hitting size of $\binom{n-1}{k}$. However, we need that the optimal egg-cut for such a scramble (i.e., the cut size) is $2\cdot(\binom{n-k}{k}-1).$ 

\begin{conjecture}
    $\KG(n,k)$ is $\lambda_2$ optimal. That is, we have that the uniform edge scramble satisfies
    $$e(\mathcal{S})=2\cdot\left(\binom{n-k}{k}-1\right).$$
\end{conjecture}

If we prove this conjecture to be true, then $\sn(\KG(n,k))=\binom{n-1}{k}$ whenever $n\geq \frac{3k^2-3k+10}{2}$. By applying Corollary \ref{cor: bounds on KG} and \Cref{prop: conj poly bound}, we have the following.\\

\loptimal*
\section{Open Questions}
Harvey and Wood conjectured that $n\geq 3k$ and $k\geq 2$ should be sufficient to show that the treewidth is $|V|-\alpha-1$. We may then ask analogous questions about the scramble number of the Kneser graphs.
\begin{question}
    What is the optimal polynomial $f(k)$ such that:
    \begin{enumerate}
        \item If $n\geq f(k)$, then $\sn(KG(n,k)) = \gon(KG(n,k)) = \binom{n-1}{k}?$
        \item If $f(k)>n>2k$, then $\sn(KG(n,k)) < \binom{n-1}{k}?$
    \end{enumerate}
\end{question}

In a similar manner, we obtain an analogous question for generalized Kneser graphs.
\begin{question}
    What is the optimal family rational of functions $f(k,t)$ such that:
    \begin{enumerate}
        \item If $n\geq f(k,t)$, then $\sn(KG(n,k,t)) = \gon(KG(n,k,t)) = \binom{n}{k}-\binom{n-t}{k-t}?$
        \item If $f(k,t)>n>2k$, then $\sn(KG(n,k,t)) < \binom{n}{k}-\binom{n-t}{k-t}?$
    \end{enumerate}
\end{question}

Another direction one may take is to study a related family of graphs, of which there are two main candidates. The first is the generalized $q$-Kneser graphs which are the ``$q$-analog" of generalized Kneser graphs. These have a known treewidth \cite{CAO2024174}, however, their scramble number and gonality is unknown. The second possible family are the Johnson graphs. These are defined in a similar way to Kneser graphs, but less is known about them regarding treewidth, scramble number and gonality.

\bibliography{references}

\end{document}